\documentclass[twoside,11pt]{article}
\usepackage{amssymb,amsmath,amsthm}           
\usepackage{hyperref}
\usepackage{eucal}
\usepackage{a4wide}
\usepackage[svgnames]{xcolor}
\usepackage{mathptmx} 
\usepackage{graphicx}
\usepackage{tikz-cd}
 \usepackage{subfigure}
 \usepackage{fancyhdr}

\title{Integrability and dynamics of the $n$-dimensional symmetric Veselova top\footnote{This research was made possible by a Newton Advanced Fellowship from the Royal Society, no.~NA140017}}

\author{Francesco Fass\`o, Luis C.~Garc\'ia-Naranjo \& James Montaldi}

\date{}

\numberwithin{equation}{section}
\numberwithin{table}{section}
\numberwithin{figure}{section}
\renewcommand{\labelenumi}{(\roman{enumi})}

\hypersetup{colorlinks=true,linkcolor=violet,citecolor=purple}

\newtheorem{theorem}{Theorem}[section]
\newtheorem{lemma}[theorem]{Lemma}
\newtheorem{proposition}[theorem]{Proposition}
\newtheorem{corollary}[theorem]{Corollary}
{\theoremstyle{definition}
\newtheorem{definition}[theorem]{Definition}
\newtheorem{remark}[theorem]{Remark}

}

\newcommand{\defn}[1]{{\bfseries\itshape{#1}}}

\def\headcolour{\color{Grey}}
\def\restr#1{\,\vrule height1.2ex width.4pt
  depth0.8ex\lower0.4ex\hbox{\scriptsize $\,#1$}}

\newcommand{\R}{\mathbb{R}}   
\newcommand{\Z}{\mathbb{Z}}   

\newcommand{\I}{\mathbb{I}}
\newcommand{\J}{\mathbb{J}}
\newcommand{\g}{\mathfrak{g}}
\newcommand{\so}{\mathfrak{so}}
\newcommand{\dd}{\mathfrak{d}}
\newcommand{\h}{\mathfrak{h}}
\newcommand{\OO}{\mathrm{O}}
\newcommand{\SO}{\mathrm{SO}}
\newcommand{\Ad}{\mathrm{Ad}}
\newcommand{\Ss}{\mathrm{S}}
\newcommand{\cB}{\mathcal{B}}
\newcommand{\xx}{\mathbf{x}}
\renewcommand{\d}{\mathsf{d}}
\newcommand{\cR}{\mathcal{R}}

\newcommand{\tr}{\mathop\mathrm{tr}\nolimits}
\newcommand{\diag}{\mathop\mathrm{diag}\nolimits}
\newcommand{\inertia}{\mathop\mathrm{in}\nolimits}

\begin{document}

\maketitle

\begin{abstract}
We consider the the $n$-dimensional generalisation of the nonholonomic Veselova problem.
We derive the reduced equations of motion in terms of the mass tensor of the body 
and determine some general properties of the dynamics. In particular we give a closed formula for the invariant measure, we
 indicate the existence of steady 
rotation solutions, and obtain some results on their stability.

We then focus our attention on bodies whose mass tensor has a specific type of symmetry. We  show that  the phase space is foliated by invariant tori that carry quasi-periodic dynamics in the natural time variable. Our results
enlarge the known  cases of integrability of the multi-dimensional Veselova top. Moreover, they show that  in some previously known
instances of integrability, the flow is quasi-periodic without the need of a time reparametrisation. 
\end{abstract}

{\small
\tableofcontents
}

\section{Introduction}

 The $n$-dimensional generalisation of some classical nonholonomic systems 
considered recently by Fedorov and Jovanovi\'c \cite{FedJov, FedJov2, Jovan, JovaRubber} provides a remarkable family of examples of
nonholonomic systems whose dynamics, after a time reparametrisation, is quasi-periodic in large dimensional invariant tori.     
A crucial  feature of these examples 
is that they possess an invariant measure and admit a Hamiltonization  by Chaplygin's multiplier method  \cite{ChapRedMult}. 
 After a time reparametrisation,  the reduced equations of motion of these examples
 become Hamiltonian and, moreover,  turn out to be
 related to classical integrable Hamiltonian systems. In this manner the authors establish the integrability of the systems and prove that
 the flow on the invariant, large dimensional tori is quasiperiodic in the new time.
 The Hamiltonization of these systems is hence central to their approach and  is truly remarkable considering that  Chaplygin's  method  is only guaranteed to work for systems with $2$ degrees of freedom. This Hamiltonization relies on the very particular type of inertia operators considered by the authors that,  as we explain below, are generally unphysical. 

In this paper we  analyse the dynamics of the $n$-dimensional  Veselova top (described below) 
and treated before by Fedorov and Jovanovi\'c \cite{FedJov,FedJov2}.  
 A fundamental difference in our study of the problem with respect to these references
 is that we assume that the inertia tensor is  {\em physical} (see below).
 
 Our main contribution is to show that under certain symmetry assumptions on the mass distribution of the body,
  the dynamics takes place on invariant tori that carry  quasi-periodic flow {\em without the need of a time reparametrisation}.
We prove this   by performing a detailed symmetry analysis of the problem, 
and in particular by performing a reduction to a system that is manifestly integrable and in fact periodic.  The quasi-periodicity of the flow of the unreduced system (in the
natural time variable) then follows from a theorem of Field  \cite{Field80}. This type of analysis to 
establish integrability of nonholonomic systems had been previously followed in e.g. \cite{Hermans, FassoGiacobbe, FS2016}.

Our approach allows us to 
recover all cases of integrability of the multi-dimensional Veselova top determined in \cite{FedJov,FedJov2} that correspond to  physical inertia tensors and we show that these 
correspond to {\em axisymmetric} bodies. We 
 also determine new cases of integrability  that we term {\em cylindrical}
 bodies. 
For this type of body,  the system does not seem to allow a Chaplygin Hamiltonization (Remark~\ref{rmk:hamiltonization}), and the generic motion   
takes place in 4-dimensional invariant tori.

This paper also considers some properties of the motion of the general $n$-dimensional Veselova top  (without any symmetry assumptions on the body).
 We prove the existence of 
 steady rotation solutions that are periodic solutions with constant angular velocity, and determine some of their properties.  In particular, we prove that some of these solutions are stable. 
 The existence of these solutions appears not to have been observed before, not even in the 3D case. 
 We also give the explicit form of 
  the reduced equations of motion   in terms of the mass tensor of the 
 body, and  a closed formula for the known invariant measure.

\paragraph{The $n$-dimensional Veselova top.}
The 3D Veselova top, introduced by Veselova in her thesis \cite{Veselova},  is a rigid body rotating under its own inertia (like Euler's rigid body) and subject to a nonholonomic constraint 
which forces the projection of the angular velocity to  a distinguished axis fixed in space, to vanish at all time (see also \cite{Veselov-Veselova-1988}).

This paper deals with a multi-dimensional generalization of this model that was introduced by Fedorov and Kozlov \cite{FedKoz}. They  consider the motion of
an $n$-dimensional rigid body subject to the nonholonomic constraint that requires the angular velocity matrix $\Omega \in \so(n)$ to have rank two and to define a rotation in a plane containing a fixed axis in space. 

The Veselova top and its $n$-dimensional generalisation are examples of {\em LR systems} which form a remarkable class of nonholonomic systems possessing an invariant measure \cite{Veselov-Veselova-1988}.

\paragraph{The inertia tensor.}
Throughout this paper, we work with the commonly accepted $n$-dimensional generalisation of  rigid body dynamics (see e.g. \cite{Ratiu80}), 
where the  inertia tensor is a linear operator $\I:\so(n)\to \so(n)$ of the form
\begin{equation}
\label{eq:Phys-Inertia-Intro}
\I(\Omega) = \J \Omega + \Omega \J,
\end{equation}
where $\J$ is the \defn{mass tensor}. 
$\J$ is a constant  $n\times n$ matrix that depends on the mass distribution of the body which, by an appropriate choice of a body frame,
 may be assumed to be   diagonal  with positive entries (see Section\,\ref{sec:rigid body} for details).  We say that a linear operator  $\I:\so(n)\to \so(n)$ is  a \defn{physical inertia tensor}
if and only if it may be written as above for a certain mass tensor $\J$.

\subsection*{Summary of results}

The configuration space of the $n$-dimensional  Veselova top problem is the orthogonal group\footnote{Readers may be more familiar with using the connected component $\SO(n)$ as the configuration space, however allowing both components simplifies our exposition in Section~\ref{sec:cylindrical}; moreover Arnold \cite[p.133]{MMCM} suggests that $\OO(n)$ is the `correct' configuration space.} 
 $\OO(n)$. The phase space for the  $n$-dimensional  Veselova top is a subbundle $D\subset T\OO(n)$ of rank $n-1$ 
determined by the nonholonomic constraints.

Our first contribution is to show that the  system admits steady rotation solutions 
on $D$. These are periodic solutions where the body 
steadily rotates in a principal plane of the body.  Because of the constraint, the orientation of the body
along these motions  is such that 
the distinguished axis is contained in the plane of rotation. Along these solutions, the nonholonomic constraint forces vanish.

Next we perform the reduction   of the system by the symmetry group $G_L=\OO(n-1)$ that corresponds
to invariance of the system under rotations and reflections of the space frame that fix the distinguished axis  (the subscript $L$ indicates
that this action is by left multiplication in the configuration space $\OO(n)$). As 
shown 
 in  \cite{FedJov,FedJov2},
this is a generalised Chaplygin reduction and the \defn{first reduced space} $D/G_L$ 
is isomorphic to the cotangent bundle $T^*\Ss^{n-1}$. We give
 the reduced equations of motion and the explicit expression for the invariant measure in terms of the mass tensor $\J$.
Moreover, with the reduced system at hand, we are able to show that the steady rotation solutions
 correspond to the singularities of an energy-momentum map
and we show stability of those occurring on {\em extremal } planes. In particular this proves that  in 3D, the steady rotations 
about the smallest and largest axes of inertia of the body are stable.

We then focus on the  study of symmetric $n$-dimensional
rigid bodies whose mass tensor has just 2 distinct eigenvalues:
$$\J = \diag[J_1,\dots,J_1,J_2,\dots,J_2],$$
with $J_1\neq J_2$. We assume that the multiplicities of $J_1$ and $J_2$ are  $r$ and  $r'=n-r$, with  $1\leq r\leq r'$. 
Under this hypothesis,
the mass tensor $\J$ has symmetry $G_R:=\OO(r)\times\OO(r')\subset\OO(n)$: if
$$h=\left[\begin{matrix}
h_{11}&0\cr 0&h_{22}
\end{matrix}\right],$$
with $h_{11}\in\OO(r)$ and $h_{22}\in\OO(r')$, then $h\J h^T=\J$. The subindex $R$ indicates that $G_R$ acts by right multiplication
on the configuration space $\OO(n)$. Since this action commutes with the left multiplication action of $G_L=\OO(n-1)$, it
 passes down to a proper action on $T^*\Ss^{n-1}$ that turns out not to be free.
 
The  \defn{second reduced space} $\cR=T^*\Ss^{n-1} /G_R$ is not a smooth manifold, but a stratified space where,  as we shall show, the dynamics is periodic.
We apply a  theorem of  Field \cite{Field80} to conclude that the reconstructed motion, both on $D$ and on $T^*\Ss^{n-1}$,
 is quasi-periodic on invariant tori. We emphasise
that this proves quasi-periodicity of the flow without a time reparametrisation.

In order to estimate  the dimension of the invariant tori on $D$ and on $T^*\Ss^{n-1}$ it is necessary to determine the details of the
stratification. For this,  we distinguish two cases depending on the multiplicity $r$
of the eingenvalue $J_1$ as follows.
\begin{enumerate}
\item[$r=1$.] The mass tensor is 
 $$\J = \diag[J_1,J_2,\dots,J_2],$$
 that corresponds to an   \defn{axisymmetric body}. In this case the orbit space $\cR$ is isomorphic to a singular semi-algebraic
 subspace in $\R^3$. We show that the dynamics on $D$ is essentially that of the 3D axisymmetric
 Veselova top (see Theorem~\ref{thm:nD to 3D} for a precise statement), 
 and is generically quasi-periodic on tori of dimension $2$ both in $D$ and on $T^*\Ss^{n-1}$. We also give a physical
 description of the dynamics.
 \item[$r>1$.] A body  with this type of symmetry is called  a \defn{cylindrical body} and can only exist in dimension 4 or higher.
  For these bodies the 
  orbit space  $\cR$ is isomorphic to a  singular semi-algebraic
 subspace in $\R^4$.
  For $n>4$ we prove that the  system evolves as a 4D cylindrical Veselova top (a precise statement is given in Theorem~\ref{thm:nD to 4D}) and
  we prove that the generic motion is quasi-periodic on tori of dimension 4 on $D$ and of dimension 3 on $T^*\Ss^{n-1}$.
  We also give an argument to prove that the generic motions do not take place in lower dimensional invariant tori.
\end{enumerate}

\paragraph{Previous work.} Fedorov and Jovanovi\'c  \cite{FedJov,FedJov2} show that  the $G_L$-reduction of the $n$-dimensional 
Veselova top to $T^*\Ss^{n-1}$  is Hamiltonizable and integrable under the assumption that the action of the
 inertia operator on  rank 2 matrices\footnote{For $a, b\in \R^n$ we denote $a\wedge b =ab^T-ba^T \in \so(n)$.} $a\wedge b\in \so(n)$
satisfies
\begin{equation}
\label{eq:inertia-Fed-Jov}
\I (a\wedge b) = (Aa)\wedge (Ab),
\end{equation}
for a diagonal matrix $A=\diag[A_1,\dots, A_n]$. In particular, this implies that the  space of rank 2 matrices 
in $\so(n)$ is invariant 
under $\I$. This condition is vacuous in dimension $3$ since every non-zero matrix in $\so(3)$ has rank 2. 
  However, for $n\geq 4$ it is a very restrictive assumption,
as we discuss in Appendix~\ref{A:Inertia}. It turns out that an inertia tensor satisfying \eqref{eq:inertia-Fed-Jov} is
physical as defined by Equation~\eqref{eq:Phys-Inertia-Intro} if and only if the body is axisymmetric. In this case $A=\diag[A_1,A_2,\dots, A_2]$,
and the mass tensor $\J$ is determined by the relations $A_2^2=2J_2$ and $A_1A_2=J_1+J_2$.

As mentioned above, their approach only allows them to conclude quasi-periodicity of the flow after a time reparametrisation.

\subsection*{Structure of the paper} 
In Section\,\ref{sec:Veselova} we begin by recalling the kinematics of the rigid body in arbitrary dimensions, governed by a mass tensor $\J$.
We then recall from \cite{FedKoz} the $n$-dimensional generalisation of the Veselova top and we discuss its symmetries, which depend on $\J$. We also describe the steady rotations which play an important role later. 
In Section~\ref{sec:1st reduction} we perform the reduction of the system by $G_L= \OO(n-1)$. We obtain the reduced equations
of motion in terms of $\J$ and determine general properties of the motion. Our analysis of axisymmetric and cylindrical bodies is respectively contained in Sections~\ref{sec:axisymmetric}  and \ref{sec:cylindrical}. In Section~\ref{S:Question-Hamiltonisation} we present some conclusions and
open questions.

The paper ends with  two appendices. The first recalls a well-known result about group actions, the reconstruction theorem of Field which we use in the proofs of quasiperiodicity, as well as some details on isotropy subgroups. The second  appendix considers in detail the physical implications of the inertia tensor hypothesis \eqref{eq:inertia-Fed-Jov} that Fedorov and Jovanovi\'c make 
in \cite{FedJov, FedJov2}.

\section{The Veselova system}
\label{sec:Veselova}

\subsection{The $\boldsymbol n$-dimensional rigid body} \label{sec:rigid body}
 
 We begin by recalling the model of an $n$-dimensional rigid body $\cB$ that moves in $\R^n$ about a fixed point $O$; further details can be found in \cite{Ratiu80,FedKoz} and the more recent work \cite{Izosimov}.  The configuration of such a body is given by an element
$g\in\OO(n)$ called the attitude matrix  that relates
an inertial frame in $\R^n$ with a body frame rigidly attached to $\cB$, with both frames having their origin in $O$. 
The velocity is given by the derivative $\dot g \in T_g\OO(n)$.  Given a motion of the body, a  material point $\xx\in\cB$ moves along a curve $t\mapsto g(t)\xx\in\R^n$, whose velocity at time $t$ is $\dot g(t)\xx$. If the point has mass $m$ then its kinetic energy
 is $\frac12m\|\dot g\xx\|^2$.
We make use of the left-trivialization of the tangent bundle of $\OO(n)$: 
\begin{equation}\label{eq:left trivialization}
\begin{array}{rcl}
T\OO(n) & \longrightarrow & \OO(n)\times\so(n) \\
(g,\dot g) &\longmapsto & (g,\Omega)
\end{array}
\end{equation}
where $\Omega=g^{-1}\dot g$. The kinetic energy of the motion of the point $\xx$ above becomes
$$\frac12 m\|\dot g\xx\|^2 = \frac12 m\|\Omega\xx\|^2.$$ 

If the body $\cB$ is a finite collection of particles, then the total kinetic energy is a sum over the constituent particles, while if it is a continuum this becomes an integral; in what follows we write it as an integral. This is conveniently done by  introducing the  \defn{mass tensor}  of the body 
$$\J = \int_\cB\xx\xx^T\d m(\xx),$$
which is a  symmetric $n\times n$ matrix. We will assume the body's geometry to be such that $\J$ is positive definite.  The kinetic energy, and hence the Lagrangian,  $L:T\OO(n)\simeq\OO(n)\times\so(n)\to\R$, may be written in terms of the mass tensor as
\begin{equation*}
L(g,\Omega) = \tfrac12\tr(\Omega\J\Omega^T).
\end{equation*}
Equivalently, we may write
\begin{equation}
\label{eq:Lagrangian}
L(g,\Omega) = \tfrac12 \langle M, \Omega\rangle ,
\end{equation}
where the \defn{angular momentum} in the body frame $M\in \so(n)$ is given by
\begin{equation}\label{eq:physical inertia}
M=\I(\Omega) = \J\Omega+\Omega\J, 
\end{equation}
and $\langle \cdot , \cdot \rangle$ denotes the 
usual invariant pairing on $\so(n)$  
\begin{equation}
\label{eq:pairing}
\left<\Lambda,\Omega\right> = \tfrac12\tr(\Lambda^T\Omega),
\end{equation}
that we use to identify $ \so(n)^*$ with $ \so(n)$.
The  symmetric, positive definite, linear map $\I:\so(n)\to\so(n)$   defined by  \eqref{eq:physical inertia}
is the  \defn{inertia tensor} (the Legendre transform). We will sometimes call $\I$ a  \defn{physical inertia tensor} 
when we wish to emphasise that it is not a generic symmetric and positive definite, linear operator on $\so(n)$, but rather one that 
 arises from the existence of a mass tensor $\J$ as in  \eqref{eq:physical inertia}. As explained in the introduction, 
 several works in nonholonomic mechanics \cite{FedJov, FedJov2, Jovan, JovaRubber,Jova18,Gajic} depend on assumptions on $\I$  that are 
 generally incompatible with it being physical.

A basis $\{f_1,\dots,f_n\}$ of the body frame that diagonalizes the symmetric mass tensor $\J$  is known as a \defn{principal basis}, and for such a basis we write
$$\J = \diag[J_1,J_2,\dots,J_n].$$
The quantities $J_i+J_j$ are called the \defn{principal moments of inertia} of the body.  Suppose $a,b$ are linearly independent vectors of $\R^n$: it is not hard to show that $a\wedge b$ is an eigenvector of the inertia operator $\I$ if and only if the plane $\Pi$ spanned by $a$ and $b$ is invariant under $\J$.  Such planes are called the \defn{principal planes} of the body, and it follows from $\J$ being symmetric that such a plane must contain two eigenvectors of $\J$, and if $J_i$ and $J_j$ are the corresponding eigenvalues, then $\I(a\wedge b)=(J_i+J_j)(a\wedge b)$. Given such a principal plane $\Pi$ we say the \defn{moment of inertia} of that plane is $\inertia(\Pi)=J_i+J_j$.  If the eigenvalues of $\J$ are simple then there are $\frac12 n(n-1)$ such principal planes, but if they are not simple then there are infinitely many.

\subsection{The Veselova constraint} \label{sec:constraint}

An $n$-dimensional version of the Veselova constraint was introduced by Fedorov and Kozlov \cite{FedKoz} and proceeds as follows.  Choose a fixed axis in $\R^n$ (space), and the (generalized) Veselova constraint allows only rotations in planes containing that axis  Note that   \cite{FedJov}  considers 
more general  constraints. In the notation of that reference, we only consider the case $r=1$.  We call this axis the \defn{distinguished axis}. 

In terms of the attitude matrix, let $e_1$ be a unit vector parallel to the distinguished axis, and complete to an orthonormal basis $\{e_1,\dots,e_n\}$ for $\R^n$ (space).   
Fix a principal basis for the body frame $\{f_1,\dots,f_n\}$, and as before denote by $g\in \OO(n)$ the attitude matrix of the body and by $\Omega = g^{-1}\dot g \in \so(n)$ its angular velocity expressed in these body coordinates. The matrix $\Omega_s:=g\Omega g^{-1}$ describes the angular velocity in the space frame, and the
 $n$-dimensional  Veselova constraint is that $\Omega_s$  is of the form $e_1\wedge w$ for some $w\in\R^n$, or equivalently 
\begin{equation}
\label{eq:constraint-qv}
\Omega=q\wedge v \qquad \mbox{for some} \quad v\in\R^n, \quad \mbox{where} \quad q:=g^{-1}e_1.
\end{equation}

Let us write
$$\dd = e_1\wedge\R^n =\mbox{span}\{ e_1\wedge e_2 , \dots, e_1\wedge e_n\}.$$  
The nonholonomic constraint distribution $D\subset T\OO(n)$ is then defined by $\Omega_s\in\dd$, or  $\dot g \in D_g=\dd g\subset T_g\OO(n)$.  By construction, $D$ is right invariant
and is non-integrable since $\dd$ is not a subalgebra of $\so(n)$. A vector $\dot g=g\Omega  \in D_g$ if and only if
$\Ad_g\Omega \in \dd$ which is equivalent to having
\begin{equation}
\label{eq:constraints}
\langle E_r\wedge E_s , \Omega \rangle =0, \qquad 2\leq r <s \leq n,
\end{equation}
where $E_r=g^{-1}e_r$ so that $\Ad_{g^{-1}}( e_r\wedge e_s)=E_r\wedge E_s$.

The left invariance of the Lagrangian \eqref{eq:Lagrangian}, together with the right invariance of the constraint distribution, signifies that this is an  example of an 
\emph{LR system} as introduced by Veselova \cite{Veselova}. Following this reference, we write the  equations of motion for the system in Euler-Poincar\'e-Arnold form:
\begin{equation}
\label{eq:VesEqns}
\begin{split}
  \frac{d}{dt} \left ( \I(\Omega) \right ) \ &= \  [\I (\Omega) , \Omega ] \ + \ \sum_{2\leq r <s \leq n} \lambda_{rs} E_r\wedge E_s, \\
  \frac{d}{dt} \left (E_r\wedge E_s \right )  \ &= \  [E_r\wedge E_s , \Omega ].
\end{split}
\end{equation}
Here,  $[\cdot , \cdot ]$ denotes the matrix commutator in $\so(n)$,  and the multipliers $\lambda_{rs}$ are  the unique solution to the linear system of equations
\begin{equation}
\label{eq:ConstVesEqns}
 \sum_{2\leq r <s \leq n}\langle \I^{-1} ( E_\rho \wedge E_\sigma ), E_r\wedge E_s   \rangle     \lambda_{rs} = -\langle  \I^{-1} ( E_\rho \wedge E_\sigma ) , [\I (\Omega) , \Omega ]  \rangle, \qquad 2\leq \rho <\sigma \leq n.
 \end{equation} 
This  choice of $\lambda_{rs}$ guarantees that $\langle E_r\wedge E_s , \Omega \rangle$ are first integrals of \eqref{eq:VesEqns} for all $2\leq r <s \leq n$. The dynamics of the Veselova system on $D$ is determined by considering the system \eqref{eq:VesEqns} restricted to the invariant manifold defined by  \eqref{eq:constraints}, 
together with the kinematical equation $\dot g=g\Omega$.

Both the restriction of  \eqref{eq:VesEqns} to  \eqref{eq:constraints},   and the system on  $D$, preserve the energy $H=\frac{1}{2}\langle \I (\Omega), \Omega \rangle$, and possess an
invariant measure \cite{Veselova}.

\paragraph{Symmetries}  Let $G_L=\OO(n-1)$ be the subgroup of $\OO(n)$ that fixes the distinguished axis, and  let $G_R\subset\OO(n)$ 
 be the group of symmetries of $\J$, that is 
\begin{equation}
\label{eq:def-G_R}
g\in G_R \qquad \mbox{if and only if} \qquad g\J g^{-1}=\J.
\end{equation}
The natural action of $G=G_L\times G_R$ on the configuration space is, for $(A,B)\in G_L\times G_R$, 
$(A,B)\cdot g = A\,g\,B^{-1}$. The tangent lift of this to the tangent bundle is 
$$(A,B)\cdot (g,\dot g) = (A\,g\,B^{-1},\;A\,\dot g\,B^{-1}).$$
Equivalently, using the left trivialization \eqref{eq:left trivialization},
\begin{equation}\label{eq:lifted action}
(A,B)\cdot (g,\Omega) = (A\,g\,B^{-1},\;B\,\Omega\,B^{-1}).
\end{equation}
It is clear that the Lagrangian \eqref{eq:Lagrangian} is invariant under this action, and furthermore, because $G_L$ fixes the distinguished axis, the constraint distribution $D$ is also invariant under this action.   Since the only ingredients that determine the dynamics of the Veselova top are the Lagrangian and the distribution,  it follows that the 
dynamics of the system is equivariant with respect to the action of $G=G_L\times G_R$.

\subsection{Steady rotations}
\label{sec:Steady-rotations}

 We now show the Veselova system admits \defn{steady rotation solutions}, which we define to be motions with constant angular velocity. 
We show such motions are necessarily periodic solutions in which the body steadily rotates  in a principal plane of inertia of the body
and the constraint forces vanish.  In view of the nonholonomic constraint, the orientation of the body along these solutions is such that the 
principal plane of rotation contains the distinguished axis.

We recall that the $n$-dimensional free rigid body also allows steady rotation solutions, and that these are in general quasi-periodic  if $n>3$. The periodicity 
of the steady rotations of the $n$-dimensional Veselova system is a consequence of the nonholonomic constraint that forces the angular velocity matrix to have rank 2.

\begin{proposition} \label{prop:steady rotations on D}
The $n$-dimensional Veselova system admits steady rotation solutions. These solutions are all steady rotations in principal planes, and are  characterised by the condition that the angular velocity matrix $\Omega$ satisfies $ [\I (\Omega) , \Omega ]=0$.  
 The steady rotations in the principal plane $\Pi$ are periodic with period  $2\pi/\|\Omega\|$, energy $H=\frac{1}{2}\inertia(\Pi)\|\Omega\|^2$ and momentum $M$ satisfying $\|M\|^2=(\inertia(\Pi))^2\|\Omega\|^2$.  Here
$\| \cdot \|$ denotes the norm in $\so(n)$  induced by the pairing \eqref{eq:pairing}.
\end{proposition}

It may seem the definition of steady rotation is ambiguous since it was not specified whether it is the angular velocity in the body or in space that should be constant.  In fact the two are equivalent: write $\Omega$ and $\Omega_s$ for the two angular velocities.  If $\Omega$ is constant, then the motion is given by $g(t)=g_0\exp(t\Omega)$ for some $g_0\in\OO(n)$. Then $\Omega_s=\dot g g^{-1} = g_0\Omega g_0^{-1}$, which is constant.  The argument is easily reversed to show the equivalence. 

\begin{proof}
First we show the existence of such solutions. 
If $\Omega$ is a constant rank two matrix that defines a rotation on a principal plane then $[\I(\Omega),\Omega]=0$ and the multipliers $\lambda_{rs}$ defined  by    \eqref{eq:ConstVesEqns} vanish, so the first equation in  \eqref{eq:VesEqns}
holds since $\Omega$ is constant.  
The motion on $\OO(n)$ is then $g(t)=g_0\exp (\Omega t)$ for a constant  $g_0\in \OO(n)$, and is periodic with the stated period since $\Omega$
has rank 2. The formulas for $H$ and $\| M\|^2$ follow immediately from $\I(\Omega)=\inertia(\Pi) \Omega$.

To complete the proof that $(g(t),\Omega)$ is a solution we should check that the nonholonomic constraints are satisfied along the motion. 
We have $\Ad_{g(t)}\Omega= \Ad_{g_0}\Omega$
so the constraints are indeed satisfied provided that the initial orientation of the body $g_0$ satisfies $\Ad_{g_0}\Omega \in \dd$. This is accomplished in the following way. If 
$\Omega=a\wedge b$  then
 $g_0\in \OO(n)$ should be such that $g_0^{-1}e_1$ lies on the plane spanned by $a$ and $b$; that is, the distinguished axis should lie on 
 the plane of rotation.

Next, if $(g(t),\Omega(t))$ is a solution for which $[\I(\Omega(t)),\Omega(t)]=0$, then  the multipliers $\lambda_{rs}$ defined  by    \eqref{eq:ConstVesEqns} vanish, and the first equation in  \eqref{eq:VesEqns}
implies that $\Omega$ is constant.

Suppose now that $(g(t),\Omega)$ is a solution with constant $\Omega \neq 0$.
It remains to show that $\Omega$ defines a rotation  on a principal plane (the condition that $[\I(\Omega),\Omega]=0$ follows immediately from this).
From  \eqref{eq:constraint-qv}, $\Omega=q\wedge v$ for some $v \in\R^n$ which can be chosen to be orthogonal to $q$.
Recall that $q=g^{-1}e_1$. Then differentiating shows that $\dot q = -g^{-1}(g\Omega)q = -(q\wedge v)q=v$. On the other hand, given that $\Omega$ is constant
we have
\begin{equation*}
0=\dot \Omega=  \dot q \wedge v + q\wedge \dot v = q\wedge \dot v,
\end{equation*}
which shows that $\dot v$ is parallel to $q$. These observations, together with the boundedness of $q$, imply that 
the vectors $q(t)$ and $v(t)$ describe simple harmonic motion on the plane that they span, which remains constant throughout the motion.
 We will now prove that this plane  is invariant under $\J$ which is equivalent to $\Omega$ defining a rotation on a principal plane. 

Using the definition of $\I$ given in \eqref{eq:physical inertia}, one shows that $[\I(\Omega),\Omega] = [\J,\Omega^2]$, and with $\Omega=q\wedge v$ this becomes
\begin{equation}\label{eq:commutator-steady rot}
[\I(\Omega),\Omega] = -\|v\|^2(\J q q^T - qq^T\J) -(\J vv^T-vv^T\J).
\end{equation}
Since $\Omega$ is constant, then so is $\I(\Omega)$, and the first of the equations of motion \eqref{eq:VesEqns} becomes
$$0 \ = \ [\I(\Omega),\Omega] \mod q^\perp\wedge q^\perp,$$
where $q^\perp$ is the subspace orthogonal to $q$. Applying both sides of the above relation to $q$ and using \eqref{eq:commutator-steady rot} gives
\begin{equation}\label{eq:Jq-steady rotation}
0 \ = \ -\|v\|^2\J q - \|v\|^2(q\cdot\J q)q+(q\cdot \J v)v,
\end{equation}
where $\cdot$ is  the Euclidean scalar product in $\R^n$. 
This shows that $\J q$ is in the span of $q$ and $v$ for all time.  The harmonic dynamics
of $q(t)$ and $v(t)$ described above imply that the same is true about $\J v$.
\end{proof}

\section{First reduction of the general Veselova top}
\label{sec:1st reduction}

We now perform the reduction of the system by the group $G_L=\OO(n-1)$ introduced above. This  corresponds to invariance of 
the system with respect to the rotations and reflections
of the space frame that fix the distinguished axis. This reduction was first considered in  \cite[Section 5]{FedJov} in a more general framework. 
Here we specialise some of their results and give their expressions  
 in terms of the mass tensor $\J$  of the body, which is necessary for our purposes.
The  geometry is clarified by  the decomposition introduced in \cite{FedJov}, $\so(n)= \dd \oplus \h$ where   $\dd$ is defined above 
and $\h$ is given by
$$\h = \mbox{span}\{ e_r\wedge e_s \, : \, 2\leq  r< s\leq n\},$$
where $e_1$ is a unit vector along the distinguished axis. 
This is the orthogonal complement of $\dd$ with respect to the Killing metric given above.
Note moreover that $\h$ is the Lie algebra of the subgroup $G_L=\OO(n-1)$. 

The splitting $\so(n)=\dd\oplus \h$ implies that we have a generalized Chaplygin system  \cite{Koi,BKMM}.
The reduced space $D/\OO(n-1)$ is isomorphic to the tangent bundle $T\Ss^{n-1}$. An isomorphism may be constructed by viewing $D$ as the horizontal space of a connection on the 
 principal $\OO(n-1)$-bundle  $\OO(n)\to \Ss^{n-1}$, as described in \cite{Koi}.  Explicitly, this map $D\to T\Ss^{n-1}$ is given by 
 $$(g,\Omega) \longmapsto (q,v)$$
where $q=g^{-1}e_1$ and $v$ is the unique vector orthogonal to $q$ satisfying 
\begin{equation}\label{eq:Omega and v}
\Omega=q\wedge v.
\end{equation}
Note that with this definition, $(q,v)$ satisfies $\|q\|=1$ and $q\cdot v=0$ showing that $(q,v)$ is indeed in $T\Ss^{n-1}$.  

The Legendre transform provides a further isomorphism of $T\Ss^{n-1}$ with $T^*\Ss^{n-1}$ and in our case  is given  as follows.

 \begin{proposition}\label{prop:Legendre}
For the Lagrangian given in \eqref{eq:Lagrangian}, the Legendre transform $T\Ss^{n-1}\to T^*\Ss^{n-1}$ is given by
$(q,v) \longmapsto (q,p)$, where
\begin{eqnarray} \label{eq:Legendre}
p=-\I(\Omega)q.
\end{eqnarray}
and $\Omega$ is given in \eqref{eq:Omega and v}.
\end{proposition}

Note that since $\I(\Omega)$ is skew symmetric, it follows that $p\cdot q=0$, which means  $T^*\Ss^{n-1}$ is realized as the submanifold of $\R^{2n}$ defined by the conditions
\begin{equation}
\label{eq:qp-conds}
\|q\|=1, \qquad p\cdot q=0.
\end{equation}
 
 \begin{proof}
The fibre derivative of the Lagrangian 
$$L(q,v) = \frac12\left<\I(q\wedge v),(q\wedge v)\right>$$
in the direction $w$ is
$$\frac12\left<\I(q\wedge v),(q\wedge w)\right>+\frac12\left<\I(q\wedge w),(q\wedge v)\right>.$$
By the symmetry of $\I$ this is equal to 
$\left<\I(q\wedge v),(q\wedge w)\right>$. Writing $\Omega=q\wedge v$ this is
\begin{eqnarray*}
-\frac12\tr\left(\I(\Omega)(qw^T-wq^T)\right) &=& 
-\frac12\tr\left(\I(\Omega)(qw^T)\right) + \frac12\tr\left(\I(\Omega)(wq^T)\right) \\
&=& -\frac12\tr\left(w^T\I(\Omega)q\right) + \frac12\tr\left(q^T\I(\Omega)w\right) \\
&=& -\I(\Omega)q\cdot w.
\end{eqnarray*}
This shows that $p=-\I(\Omega)q$ as required.
 \end{proof}

Composing the two isomorphisms gives 
$$D/\OO(n-1)\longrightarrow T^*\Ss^{n-1},\quad (g,\Omega) \longmapsto (q, p),$$
where $q=g^{-1}e_1$ and $p=-\I(\Omega)q$. This map is easily shown to be well-defined. 

Following the action of $G=G_L\times G_R$ on $D$ (see Section\,\ref{sec:constraint}) through the reduction process above it is easy to see that,

\begin{corollary}\label{coroll:symmetry on T^*S}
The symmetry $G=G_L\times G_R$ of the Veselova top descends to the natural action of $G_R$ on the spaces $T\Ss^{n-1}$ and $T^*\Ss^{n-1}$; that is $g\in G_R$ acts by $g\cdot(q,v) = (gq,gv)$ and $g\cdot(q,p) = (gq,gp)$.
\end{corollary}
The following lemma is a particular instance of  \cite[Theorem 5.4]{FedJov}:
\begin{lemma}\label{lemma:q,p constant}
The vectors $q$ and $p$ are constant in the space frame, that is, they satisfy
\begin{equation}\label{eq:qdot}
\dot q=-\Omega q,\qquad \dot p = -\Omega p.
\end{equation}
\end{lemma}

\begin{proof}
Since $q=g^{-1}e_1$, in space it becomes $e_1$ which is fixed by definition (alternatively, differentiating $q=g^{-1}e_1$ shows $\dot q = -\Omega q$).  Next, starting with the relation $p=-\I(\Omega)q$ and using \eqref{eq:VesEqns} shows that
\begin{equation} \label{eq:pdot}
\dot p =\left ( -\frac{d}{dt} (\I(\Omega)) + \I(\Omega) \Omega  \right ) q = \Omega  \I(\Omega) q = -\Omega p,
\end{equation}
where we have used $(E_r\wedge E_s)q=0$ for all $2\leq r <s \leq n$.
\end{proof}

Combining the first equation in~\eqref{eq:qdot} with Equation\,\eqref{eq:Omega and v} shows that along the motion
\begin{equation}\label{eq:Omqqdot}
\Omega= q\wedge \dot q.
\end{equation}
We will use this repeatedly in what follows.

 Let $C$ be the $n\times n$ matrix defined by,
\begin{equation}\label{eq:Cmatrix}
C=C(q) = \left(\J + (q\cdot\J q)\mathrm{Id}_n\right)^{-1},
\end{equation}
where $\mathrm{Id}_n$ denotes the identity matrix. (It is easy to see the matrix in brackets is diagonal and positive definite and hence invertible.)
The following proposition  gives the explicit form of the reduced equations for a general physical inertia tensor (i.e., one derived from a mass tensor $\J$). 

\begin{proposition}\label{prop:eqns-nD} 
The reduced equations of motion on $T^*\Ss^{n-1}$ are
\begin{eqnarray}
\label{eq:RedEqnsGeneral}
\left\{\begin{aligned}
\dot q &= C\left[ p - \left(\frac{p\cdot Cq}{q\cdot Cq}\right)q\right],    \\
\dot p &=-2H(q,p)q.
\end{aligned}\right.
\end{eqnarray}
Here,   
\begin{equation}
\label{eq:energy-general}
H(q,p)=\frac{1}{2} \left (p\cdot Cp -\frac{(p\cdot Cq)^2}{ q\cdot Cq }\right ) ,
\end{equation}
is the energy  integral.  Furthermore, this system is invariant under the action of $G_R$ described in Corollary~\ref{coroll:symmetry on T^*S}.
\end{proposition}

\begin{proof}
In view of~\eqref{eq:Omqqdot} we have
\begin{equation*}
p=-\I(\Omega)q= -(\J (q\wedge \dot q) + (q\wedge \dot q) \J)q.
\end{equation*}
We would like to solve the above equation for $\dot q$ in terms of $q$ and $p$. This is not possible for general vectors in $\R^n$ since the right hand side is a linear expression in $\dot q$ that vanishes whenever $\dot q$ is parallel to $q$.  However,
on $T^*\Ss^{n-1}$ the equations  \eqref{eq:qp-conds}  hold and, since $\dot q\cdot q=0$, we may uniquely write $\dot q$ as in the first equation in  \eqref{eq:RedEqnsGeneral}.

Using that $p\cdot q=0$ along $T^*\Ss^{n-1}$,
 the equation $\dot p =-\Omega p$ (see lemma above)  simplifies to $\dot p=- (\dot q\cdot p) q=-2H(q,p) q$ which is the second equation in \eqref{eq:RedEqnsGeneral}.

It remains to show that $H(q,p)$ given by \eqref{eq:energy-general} coincides with  the energy $\frac{1}{2}\langle \I (\Omega), \Omega \rangle $. Using once again~\eqref{eq:Omqqdot} we obtain
\begin{equation*}
\tfrac{1}{2}\langle \I (\Omega), \Omega \rangle = -\tfrac{1}{2} \tr ( \J (q\wedge \dot q)^2)=\tfrac{1}{2}C^{-1}\dot q \cdot \dot q=\tfrac{1}{2}p \cdot \dot q =H(q,p).
\end{equation*}
Finally, to see the  invariance, note that \eqref{eq:def-G_R} implies that  $C(gq) = C(q)$ for all $g\in G_R$.
\end{proof}

 Lemma \ref{lemma:q,p constant} above allows us to conclude the following proposition 
 (which  is an instance of a general result on kinetic LR systems given in \cite[Proposition 2.2]{FedJov}).
\begin{proposition}\label{prop:conservation of P}
The square $\|p\|^2$ of the momentum is conserved under the dynamics on $T^*\Ss^{n-1}$.
\end{proposition}
\begin{proof}
This  is  clear from \eqref{eq:RedEqnsGeneral}, since $p\cdot \dot p = -2H p\cdot q = 0$.
\end{proof}

Another application of Lemma\,\ref{lemma:q,p constant} is the following reconstruction formula, valid for the 3-dimensional system, and found in a slightly different form in \cite{Veselov-Veselova-1988} (not using $p$). 

\begin{proposition}\label{prop:reconstruction 3D}
In dimension 3, if $(q(t),p(t))$ is a solution of the reduced equations of motion on $T^*\Ss^2$, then for any $g_0\in G_L=\OO(2)$ the curve
\begin{equation}
\label{eq:3d-reconst-g}
g(t)=g_0\begin{pmatrix}
q(t)^T\cr
\frac1{\|p(t)\|}p(t)^T \\
\frac1{\|p(t)\|}(q(t)\times p(t))^T 
\end{pmatrix}
\end{equation}
is a solution to the full Veselova system, and all solutions with $p\neq0$ can be written in this way.
\end{proposition}

\begin{proof}
We know $q,p$ are constant in space, whence $q\times p$ is also constant in space. Thus each row satisfies $\dot r = r\Omega$ and hence $\dot g = g\Omega$.  Furthermore, for such $g_0$, the velocity $\dot g$ lies in the distribution $D$.  This is because the angular velocity in space $\Omega_s$ satisfies
$$\Omega_s = \dot g(t) g^{-1}(t) = g_0
\begin{pmatrix}
0&2H/\|p\| & (*)\\
-2H/\|p\| & 0 & 0 \\
(*) & 0 & 0
\end{pmatrix}g_0^{-1}
$$
and since the lower left $2\times 2$ block vanishes, and this is preserved by $g_0$, $\Omega_s$ satisfies the constraint.
\end{proof}

\subsection{Reduced dynamics of the general Veselova top}
With the equations of motion above, we are able to discuss some aspects of the reduced dynamics of the general (physical) $n$-dimensional Veslova top.

Note that since the system on $D$ is an LR system it follows from a theorem of Fedorov and Jovanovi\'c \cite[Theorem 3.3]{FedJov} that Equations~\eqref{eq:RedEqnsGeneral}  possess a smooth invariant measure on  $T^*\Ss^{n-1}$; namely the push forward of the invariant measure 
on $D$ mentioned already. A formula for this volume form, valid in a more general framework than the one 
we consider, is given in \cite[Theorem 5.5]{FedJov}. The following proposition gives a non-trivial simplification of this formula to our setting.

\begin{proposition}\label{prop:measure-nD} 
The reduced equations of motion \eqref{eq:RedEqnsGeneral} on $T^*\Ss^{n-1}$ have invariant measure
\begin{equation*}
\sqrt{\frac{\det (C)}{q\cdot C q}} \, \sigma,
\end{equation*}
where $\sigma$ is the 
Liouville volume form on $T^*\Ss^{n-1}$.
\end{proposition}

\begin{proof}  
The Liouville form $\sigma$  on $T^*\Ss^{n-1}$ coincides with the volume form on $T^*\Ss^{n-1}$ inherited from the Euclidean measure $dq\wedge dp$ in 
the ambient space $\R^{2n}$. To prove the result we compute the divergence of the vector field on $\R^{2n}$ defined by the 
equations  \eqref{eq:RedEqnsGeneral} with respect to $\mu(q) \, dq\wedge dp$, where $\mu(q)=\sqrt{\frac{\det (C)}{q\cdot C q}}$,
 and show that its restriction to $T^*\Ss^{n-1}$ vanishes.
Given that $\mu$ is independent of $p$, this is equivalent to showing that 
\begin{equation*}
y(q,p):=  \mu\sum_{i=1}^n  \left ( \frac{\partial \dot q^i}{\partial q^i}+ \frac{\partial \dot p_i}{\partial p_i}\right ) + \frac{\partial \mu}{\partial q} \cdot  \dot q  
\end{equation*}
 is zero along $T^*\Ss^{n-1}$. In the calculations below  we will repeatedly use
that  the matrices  $\J$ and $C$ are symmetric and commute (which is obvious since they are diagonal).

First notice that \eqref{eq:RedEqnsGeneral} implies 
\begin{equation*}
 \sum_{i=1}^n   \frac{\partial \dot p_i}{\partial p_i} =   -2 \frac{\partial H}{\partial p}\cdot q = - 2\dot q \cdot q,
\end{equation*}
which vanishes along $T^*\Ss^{n-1}$. Next, by direct differentiation of $\mu$ and $\dot q^i$ defined by \eqref{eq:RedEqnsGeneral} one obtains
\begin{equation*}
\begin{split}
\frac{\partial \mu}{\partial q}& = \mu \left (- X(q) \J q - \frac{Cq}{q \cdot Cq} \right ), \\ 
  \sum_{i=1}^n  \frac{\partial \dot q^i}{\partial q^i}&= \left ( C \left (  -2\J + \frac{2(q \cdot \J Cq) -1 }{q \cdot Cq} \mbox{Id}_n \right ) q \right ) \cdot \dot q -
  X(q) \frac{q \cdot Cp}{q \cdot Cq},
 \end{split}
\end{equation*}
where $X(q):=\mbox{tr}(C) - \frac{q\cdot C^2 q}{q\cdot Cq}$. Therefore, along $T^*\Ss^{n-1}$ we may
write
\begin{equation}
\label{eq:aux-proof-inv-measure}
y(q,p)= - 2\mu  \left ( C \left (  \J + \frac{1- (q \cdot \J Cq) }{q \cdot Cq} \mbox{Id}_n \right ) q \right ) \cdot \dot  q - \mu X(q) 
 \left ( \J q \cdot \dot  q + \frac{q \cdot Cp}{q \cdot Cq} \right ).
\end{equation}

We now make use of the equality 
\begin{equation*}
\J C k \cdot q +(Ck \cdot q)(\J q \cdot q) = k\cdot q,
\end{equation*}
which holds for all $k\in \R^n$ and is easily established. By respectively applying it to $k=q$ and  $k=p$ shows that along $T^*\Ss^{n-1}$ we have
\begin{equation*}
q \cdot \J Cq =1-(q \cdot Cq) (q \cdot \J q), \qquad \mbox{and} \qquad q \cdot \J C p = -( q \cdot Cp) (q \cdot \J q).
\end{equation*}
Using these identities and the expression for $\dot q$ given in \eqref{eq:RedEqnsGeneral} one deduces the formulae
\begin{equation*}
  \J + \frac{1- (q \cdot \J Cq) }{q \cdot Cq} \mbox{Id}_n = C^{-1}, \qquad  \J q \cdot \dot  q =- \frac{q \cdot Cp}{q \cdot Cq},
\end{equation*}
that are valid along $T^*\Ss^{n-1}$. Therefore, \eqref{eq:aux-proof-inv-measure} implies
\begin{equation*}
\left . y(q,p) \right |_{T^*\Ss^{n-1}} = - 2\mu q \cdot \dot q =0.
\end{equation*}
\end{proof}

\begin{remark}\label{rmk:hamiltonization}  It
 is 
  unknown whether the reduced system   \eqref{eq:RedEqnsGeneral}  can be Hamiltonized by 
Chaplygin's reducing multiplier method. 
This would amount to the introduction  of a time reparametrization $dt = \nu(q) \, d\tau$, and a rescaling of the momenta $p\to \nu(q)p$, that would transform the
equations into a Hamiltonian system.  The factor $\nu(q)$, if it exists, is the so-called Chaplygin multiplier.
The specific form of the invariant measure implies that, up to a constant factor, the multiplier   $\nu(q)$ would be given by
\begin{equation*}
\nu(q) = \left ( \frac{\det C}{q\cdot Cq} \right )^{-\frac{1}{2(n-2)}}.
\end{equation*}
The details about the  method may be found in e.g.  \cite{FedJov,EhlersKoiller}.
 For $n=3$ the Hamiltonization is possible and follows from Chaplygin's last multiplier theorem \cite{ChapRedMult}.
For $n\geq 4$, using the above expression for $\nu(q)$, we were able to verify the Hamiltonization of the system for  axisymmetric tops (treated in Section \ref{sec:axisymmetric}). This in fact follows from the work of
 Fedorov and Jovanovi\'c  \cite{FedJov,FedJov2}. On the other hand, some preliminary  investigations of us suggest that the system is {\em not} Hamiltonizable for $n\geq 4$ and more general
mass tensor $\J$.
\end{remark}

\subsection{Steady rotations}

\begin{proposition} \label{prop:steady rotations on T^*S}
The (non-trivial) steady rotation solutions on $D$ described in Proposition \ref{prop:steady rotations on D} project to periodic solutions of the reduced system  \eqref{eq:RedEqnsGeneral} on $T^*\Ss^{n-1}$, where $q$ and $p$ rotate uniformly in a principal plane of the body.  The projected motion has the same period and the same energy as the motion on $D$.
\end{proposition}

We prove the stability of some of these steady rotations in Theorem\,\ref{thm:stability} below. 

\begin{proof}
 Let  $\J=\mbox{diag}[J_1,\dots, J_n]$ be the mass tensor and consider a steady
rotation solution with angular velocity $\Omega$.
 In view of Proposition\,\ref{prop:steady rotations on D}, by rotating the body frame if necessary, we may assume that  $\Omega=\omega f_i\wedge f_j$, where 
 $\omega = \| \Omega \| \in \R$.
 Then $p=-\I(\Omega)q=-(J_i+J_j)\Omega q=(J_i+J_j)\dot q$.
It follows from the second equation in \eqref{eq:RedEqnsGeneral} that the motion of $q$, $p$, is simple harmonic with period $2\pi\sqrt{(J_i+J_j)/2H}=2\pi/|\omega|$. 
Moreover, $q$ and $p$ are contained in the plane spanned by $f_i$ and $f_j$ in view of~\eqref{eq:Omqqdot}.
 That the energies are equal is obvious, by the definition of $H$.
\end{proof}

Consider the integral of motion  $P=\| p\|^2$ and the resulting energy-momentum map 
$$(P,H) : T^*\Ss^{n-1}\to \R^2.$$  
The following proposition refers to the singularities of this map.   Since $H$ and $P$ are both non-negative, the image lies in the (closed) positive quadrant $\R_+^2\subset\R^2$.  If $p=0$ then $(P,H)(q,0)=(0,0)$ and such points are all critical points of $(P,H)$.   As we see in the following proposition, the other critical points occur at the steady rotations.  

\begin{proposition} \label{prop:EM critical rays}
Let $\J = \diag[J_1, \dots, J_n]$ with $J_j>0$ for all $j$.  
\begin{enumerate}
\item The critical points of the energy-momentum map $(P,H)$ with $p\neq0$ occur at the steady rotations described above.
\item The critical values of the energy-momentum map $(P,H)$ are the rays $L_{ij}$ given by $P=2(J_i+J_j)H$, for $i<j$.

\end{enumerate}
\end{proposition}

The image and critical values of the energy-momentum map are illustrated in Figure\,\ref{fig:E-M for Veselova}.  Let us emphasize that, in the case of eigenvalues with multiplicity, this statement does not assume $J_i\neq J_j$, only that $i\neq j$. 
Moreover, recall from Section~\ref{sec:rigid body} that a rotation of the body in the principal plane $\Pi_{i,j}$ has angular velocity $\Omega$ equal to a multiple of $f_i\wedge f_j$, and this is an eigenvector of the inertia operator with eigenvalue $\inertia(\Pi_{i,j})=J_i+J_j$.

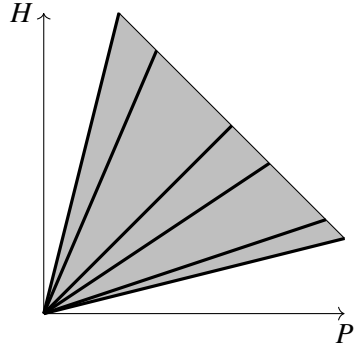
\begin{figure}
\centering
\begin{tikzpicture}[scale=2.5]
 \draw[->] (0,0) -- (1.6,0) node[anchor=north]{$P$};
 \draw[->] (0,0) -- (0,1.6) node[anchor=east]{$H$};
  \draw[fill=lightgray] (0,0) -- (0.4,1.6)  -- (1.6,0.4);
  \draw[very thick] (0,0) -- (0.4,1.6); 
  \draw[very thick] (0,0) -- (0.6,1.4); 
  \draw[very thick] (0,0) -- (1,1); 
  \draw[very thick] (0,0) -- (1.2,0.8); 
  \draw[very thick] (0,0) -- (1.5,0.5); 
  \draw[very thick] (0,0) -- (1.6,0.4); 
\end{tikzpicture}
\begin{minipage}{0.9\textwidth}
\caption{The image and critical rays $L_{ij}$ (shown as dark lines, see Proposition \ref{prop:EM critical rays}) of the energy-momentum map for the Veselova top. The slopes of the critical rays as shown are equal to $[2(J_i+J_j)]^{-1}$.} 
\label{fig:E-M for Veselova}
\end{minipage}
\end{figure}

\begin{proof}

(i) 
Let $(q_0,p_0)\in T^*\Ss^{n-1}$ with $p_0\neq0$. Writing $H_0=H(q_0,p_0)$ and $P_0=\|p_0\|^2$, this implies $H_0>0$ and $P_0>0$. At such points, neither $dH=0$ nor $dP=0$.  Thus, considering $H$ and $P$ as functions on the ambient space $\R^n\times\R^n$, the condition for a critical point at $(q_0,p_0)$ is that the gradients of $H$ and $P$ are linearly dependent (on $T^*\Ss^{n-1}$), and using Lagrange multipliers for the constraints defining this cotangent bundle gives
\begin{equation}
\label{eq:LagMult1}
\nabla H = \lambda_1 \nabla P + \lambda_2 \nabla (p\cdot q)+\lambda_3 \nabla \|q\|^2, 
\end{equation}
and $(q_0,p_0)$ are such that 
\begin{equation}
\label{eq:CondExtrema}
q_0\cdot p_0=0, \qquad
\|q_0\|^2=1, \qquad
H(q_0,p_0)=H_0, \qquad \|p_0\|^2=P_0.
\end{equation}
At such points,  \eqref{eq:LagMult1} is equivalent to 
\begin{equation}
\label{eq:LagMult1bis}
\nabla H(q_0,p_0) = 2\lambda_1(0,p_0) + \lambda_2 (p_0,q_0)+2\lambda_3(q_0,0). 
\end{equation}
Denote by
\begin{equation*}
v (q,p):=\dot q= Cp-  \frac{p\cdot Cq}{q\cdot Cq}Cq.
\end{equation*}
One calculates 
\begin{equation*}
\frac{\partial H}{\partial p} =v, \qquad \frac{\partial H}{\partial q}=-\|v\|^2 \J q - \frac{(p\cdot Cq)}{(q\cdot Cq)}v.
\end{equation*}
Then \eqref{eq:LagMult1bis} can be rewritten as 2 equations,
\begin{subequations}
\label{eq:LagMult2}
\begin{align}
\label{eq:LagMult2a}
v_0 & = 2\lambda_1 p_0 + \lambda_2 q_0, \\
\label{eq:LagMult2b}
 - \left(\frac{p_0\cdot Cq_0}{q_0\cdot Cq_0}\right)\,v_0   -\|v_0\|^2 \J q_0 &= \lambda_2 p_0 + 2\lambda_3 q_0,
\end{align}
\end{subequations}
where $v_0=v(q_0,p_0)$. 
Taking the  inner product on both sides of equation \eqref{eq:LagMult2a} 
with $q_0$ and $p_0$, and using \eqref{eq:CondExtrema}, leads to
the relations:
\begin{equation} \label{eq:Lagrange multipliers 1,2}
\lambda_2=0, \qquad \lambda_1=\frac{H_0}{P_0}.
\end{equation}
Substituting this into \eqref{eq:LagMult2b} and taking the inner product with $q_0$ gives
$\lambda_3 = -\frac12\|v_0\|^2(q_0\cdot\J q_0).$
Putting first $\lambda_2=0$ in \eqref{eq:LagMult2a} shows that $v_0=2(H_0/P_0)p_0$, and secondly taking the inner product of \eqref{eq:LagMult2b} with $p_0$ shows
$$\lambda_3=-2\frac{H_0^2}{P_0}(q_0\cdot\J q_0),\quad\text{and}\quad 
\frac{p_0\cdot Cq_0}{q_0\cdot Cq_0}=-\frac{2H_0}{P_0}(p_0\cdot \J q_0).
 $$
Now consider $\J q_0$ and $\J p_0$.  First from \eqref{eq:LagMult2a}, $v_0=(2H_0/P_0)p_0$ and then from \eqref{eq:LagMult2b} we deduce that $\J q_0$ is a linear combination of $q_0$ and $p_0$. Again from \eqref{eq:LagMult2a}, substituting for the definition of $v_0$, and using the definition of the matrix $C$ one finds
$$\J p_0 = \left(\frac{P_0}{2H_0}-(q_0\cdot\J q_0)\right)p_0 + \frac{P_0}{2H_0}\left(\frac{p_0\cdot Cq_0}{q_0\cdot Cq_0}\right)q_0,$$
which is also a linear combination of $p_0$ and $q_0$.  
In this way we have shown that the plane spanned by $q_0,p_0$ is invariant under $\J$, which is precisely the condition for it to be a principal plane. 

Finally, if $p_0,q_0$ lie in a principal plane, and $p_0\neq0$, they define a steady rotation  with $\Omega= \frac{2H_0}{P_0}q_0\wedge p_0$
as is seen from Equations~\eqref{eq:Omqqdot}, \eqref{eq:LagMult2a} and \eqref{eq:Lagrange multipliers 1,2}.

(ii) Without loss of generality, we may assume that the body frame $\{f_1, \dots , f_n\}$ is chosen such that $\J=\diag[J_1, \dots J_n]$ and
$\mbox{span} \{q_0, p_0\} =\mbox{span} \{f_i, f_j\} $ ($i\neq j$). We may then write
\begin{equation*}
q_0=\cos \alpha f_i + \sin \alpha f_j, \qquad p_0= \sqrt{P_0}\left ( -\sin \alpha f_i + \cos \alpha f_j \right ),
\end{equation*}
for a certain $\alpha \in [0,2\pi)$.  The relation $P_0=2H_0(J_i+J_j)=2H_0\inertia(\Pi_{i,j})$ follows from Proposition\,\ref{prop:steady rotations on D}, since $P=\|M\|^2$. 
\end{proof}

\paragraph{Stability of steady rotations.} Recall that for a dynamical system, a compact subset $S$  of the phase space is Lyapunov stable if
 for each neighbourhood $U$ of $S$ there is a neighbourhood $V$ such that any trajectory that enters $V$ lies entirely in $U$.  For the free $n$-dimensional rigid body, Izosimov \cite{Izosimov} has produced a fairly complete analysis 
 of the stability of the relative equilibria.  In particular, in Example 2.3 he shows that if $\J$ has simple eigenvalues, with $J_1<J_2<\dots<J_n$ then the 2-dimensional rotations in the principal plane $\Pi_{i,j}$ are (Lyapunov) stable if and only if $|i-j|=1$ (that is, if $J_i$ and $J_j$ are adjacent in the ordering).  It is not known whether the same result holds with the Veselova constraint.  However, we do have the following result which follows from the geometry of the energy-momentum map.  We do not require the eigenvalues of $\J$ to be simple (motivated by the symmetric bodies considered in later sections).

Recall that if $\Pi$ is a principal plane of inertia then there is a corresponding moment of inertia, which we denote $\inertia(\Pi)$; if $\Pi=\Pi_{i,j}$ for some principal basis then $\inertia(\Pi_{i,j})=J_i+J_j$.

\begin{definition}\label{def:extremal}
A principal plane $\Pi$ is \defn{extremal} if either $\inertia(\Pi)\leq\inertia(\Pi')$ for all principal planes $\Pi'$ (in which case it is minimal) or $\inertia(\Pi)\geq\inertia(\Pi')$  for all principal planes $\Pi'$ (in which case it is maximal).
\end{definition}

If the eigenvalues of $\J$ are all simple and ordered as usual, then the only minimal principal plane is $\Pi_{1,2}$ and the only maximal one is $\Pi_{n-1,n}$.  
If some of the eigenvalues of $\J$ are not simple, then the dynamics has a further symmetry $G_R$ described by Equation\,\eqref{eq:def-G_R}, and one does not in that case expect a steady rotation to be Lyapunov stable.  Instead the natural notion is $G_R$-Lyapunov stable, which is defined by saying that a $G_R$-invariant subset $S$ of phase space is $G_R$-Lyapunov stable if for every $G_R$-invariant neighbourhood $U$ of $S$ there is a $G_R$-invariant neighbourhood $V$ of $S$ such that any trajectory that intersects $V$ is entirely contained in $U$. 

\begin{theorem} \label{thm:stability}
Steady rotations in extremal principal planes are $G_R$-Lyapunov stable. 
\end{theorem}

If as above, the eigenvalues are simple then the statement reduces to saying the steady rotations in extremal principal planes are Lyapunov stable. 
 Before the proof of the theorem, we state the following immediate consequence.
\begin{corollary} \label{cor:stability}
The steady rotations of the  $3D$ Veselova top about the largest and smallest principal axes of inertia are Lyapunov stable
($\OO(2)$-Lyapunov stable for axisymmetric tops).
\end{corollary}

\begin{proof}[Proof of Theorem~\ref{thm:stability}]
We shall prove this for minimal principal planes of inertia, and leave the maximal case to the reader. We define a function on the phase space which, under the hypotheses of the theorem, is a Lyapunov function.  Let $(q_0,p_0) \in \Pi$ be a point of steady rotation, and let $f:T^*\Ss^{n-1}\to\R$ be defined by 
$$f(q,p) = H-\lambda P - \tfrac12(P-c)^2,$$
where $\lambda=(2\inertia(\Pi))^{-1}$ and $c>0$ is a constant to be fixed.   This function is clearly invariant under the dynamics, and is moreover invariant under the action of $G_R$.  Let $S$ be the $G_R$-orbit of the trajectory through $(q_0,p_0)$.  We show that (i) $f$ has a critical point at every point of $S$, and (ii) under the condition that $\Pi$ is minimal, the Hessian matrix of $f$ transverse to  $S$
 is negative definite. It then follows that for $\varepsilon>0$ the set
$$V_\varepsilon := \{(q,p)\in T^*\Ss^{n-1} \mid f(q,p) > f(q_0,p_0)- \varepsilon\}$$
is a $G_R$-invariant neighbourhood of $S$, and that any $G_R$-invariant neighbourhood $U$ of $S$ contains $V_\varepsilon$ for some $\varepsilon>0$, and we can deduce the $G_R$-Lyapunov stability of $S$.

For (i), since $f$ is $G_R$-invariant, it suffices to prove it has a critical point at any point of the trajectory.  Now, writing $T_0$ for the tangent space $T_0=T_{(q_0,p_0}(T^*\Ss^{n-1})$, we have
$$df = (dH-\lambda dP)_{\restr{T_0}} - (P-c)dP_{\restr{T_0}}.$$
The first term vanishes by the proof of Proposition\,\ref{prop:EM critical rays} above, and in particular the value  of $\lambda$ is given by $\lambda_1$ in that proof---see \eqref{eq:Lagrange multipliers 1,2}.  The second term vanishes when $P=c$, so from now on we assume $c=\|p_0\|^2.$

We now wish to compute the Hessian of $f$ at $(q_0,p_0)$ in directions transverse to the submanifold $S$. Let us choose a principal basis so that $\Pi=\mathrm{span}\{f_1,f_2\}$, in which case
$$q_0=\cos\alpha f_1+\sin\alpha f_2,\quad p_0=\|p_0\|(-\sin\alpha f_1+\cos\alpha f_2).$$
We can further simplify the calculations by noting that since $f$ is invariant under the dynamics, the signature of the Hessian will also be invariant (by Sylvester's law of inertia), so we can suppose $\alpha=0$, and we take
\begin{equation}\label{eq:q0 & p0}
q_0= f_1,\quad p_0=\|p_0\| f_2.
\end{equation}

Now we consider a slice $V$ (of dimension $2n-3$) to the trajectory in the tangent space $T_0$. With $(q_0,p_0)$ as above, and putting $\hat p = (0,f_2)$ (the radial direction in $p$) we can choose the following basis for $V$:
$$\left\{\hat p,\;(af_3,0),\;( 0, b_3f_3),\dots,\;(af_n,0),\;(0, b_nf_n)\right\}$$
(in $\R^n\oplus\R^n$), where $b_k = \sqrt{(J_1+J_2)(J_1+J_k)}$ and $a=\omega^{-1}$, where $\omega$ is the frequency of the steady rotation.
With respect to this basis, one finds the Hessian matrix of $f$ at $(q_0,p_0)$ to be in block form,
$$\left[\begin{array}{c|cccccc}
- 8\omega^2(J_1+J_2)^2 &0&0 &0& \dots &&0\cr
\hline
0&A_3&0&0&\dots&&0 \cr
0&0&A_4&0&\dots&&0 \cr
0 & 0&0&A_5&\dots &&0\cr
\vdots&\vdots&\vdots&&\ddots&&\vdots\cr
0&0&0&&\dots&&A_n
\end{array}\right],$$
where the $2\times2$ matrix $A_k$ is given simply by
\begin{equation}\label{eq:stability matrices}
A_k \ = \ \begin{pmatrix}J_1-J_k & 0 \\
0& J_2-J_k
\end{pmatrix}.
\end{equation}
The remainder of the proof is simply combining the signs of these eigenvalues with the subspace of $V$ that suffices for a transversal to $S$.

Firstly, note that if all the eigenvalues $J_k$ are distinct, then the Hessian matrix is negative definite if and only if $J_k>\max\{J_1,J_2\}$ for all $k\geq3$, which is precisely the condition for a minimal principal plane. Indeed this holds more generally, provided $\max\{J_1,J_2\} <\min\{J_k\mid k\geq3\}$  (whether or not $J_1=J_2$ or some of the $J_k$ for $k\geq3$ are multiple).

Now assume $J_1\leq J_2 = J_3$ (which covers all remaining cases).  

First consider the case where $J_1=J_2$ and suppose this has multiplicity $m\geq3$. Let $E_1$ be the $m$-dimensional eigenspace of $\J$ with eigenvalue $J_1$.  The tangent space to the $\OO(m)$-group orbit through $(q_0,p_0)=(f_1,\|p_0\|f_2)$ is  
$$\g\cdot(q_0,p_0) = \{(\xi f_1,\|p_0\|\xi f_2)\mid \xi\in\so(m)\}.$$
Since $m\geq 3$ it follows that 
$$\g\cdot(q_0,p_0) = \R\{(f_2,-\|p_0\|f_1)\} \oplus (F_1\times\{0\}) \oplus (\{0\}\times F_1),$$
where $F_1$ is the span of $f_3,\dots,f_m$ in $E_1$.  The first component here is the tangent space to the trajectory, so is not required in $V$, and then 
$$T_0=\R(\dot q_0,\dot p_0) \oplus V = \g\cdot(q_0,p_0) \oplus V_1$$
where $V_1$ has basis
$$\left\{\hat p,\;(af_{r+1},0),\;( 0, b_{r+1}f_{r+1}),\dots,\;(af_n,0),\;(0, b_nf_n)\right\},$$
using the notation introduced above.  It is clear that the restriction to this space of the Hessian matrix given above is again negative definite.

Finally suppose $J_1<J_2=J_3$, and let $m_2$ be the multiplicity of $J_2$, and $E_2$ the eigenspace. Since $J_1$ is simple there is only a finite group acting on that space, so that 
$$\g\cdot (q_0,p_0) = \{0\}\times E_2$$
where $E_2$ is the span of $\{f_3,\dots,f_{m+1}\}\subset E_2$. Again writing, 
$$T_0=\R(\dot q_0,\dot p_0) \oplus V = \R(\dot q_0,\dot p_0) \oplus \g\cdot(q_0,p_0) \oplus V_2$$
we have $V_2$ is spanned by $\hat p$ and $(f_r,0)$ ($r=3,\dots,m+1$) and by $(f_k,0),(0,f_k)$ for $k>m+1$. The eigenvalues of the Hessian above that are relevant are just the first (for the space spanned by $\hat p$) and the first element of each $A_k$ (for $k=3,\dots,m+1$, which are all negative), and the remaining $2\times 2$ blocks, which are all negative definite. 
\end{proof}

\section{The axisymmetric Veselova top}  \label{sec:axisymmetric}

This section considers \defn{axisymmetric tops} for which the mass tensor  relative to a frame $\{f_1, \dots f_n\}$ takes the form
\begin{equation}\label{eq:axisymmetric J}
\J = \diag[J_1,J_2,\dots,J_2].	
\end{equation}
From its definition in~\eqref{eq:def-G_R},  it follows that the symmetry group $G_R=\OO(n-1)$, which corresponds
to rotations and reflections of the body frame that fix the symmetry axis.

In order to analyse the  \defn{second reduced space} $\cR:=D/G = T^*\Ss^{n-1}/G_R$, we require the invariants for the action of $G_R=\OO(n-1)$ on $T^*\Ss^{n-1}$.  
Now, it is well-known that the ring of invariants for $\OO(n-1)$ acting on $(q_2,p_2) \in \R^{n-1}\times\R^{n-1}$ is generated by three independent quadratic expressions:
$$\|q_2\|^2,\quad \|p_2\|^2,\quad\text{and}\quad q_2\cdot p_2.$$
By the Cauchy-Schwarz inequality, these satisfy $|p_2\cdot q_2|^2 \leq\|q_2\|^2\|p_2\|^2$.

It follows that the ring of invariants for $\OO(n-1)$ acting on $(q_1,q_2,p_1,p_2) \in \R\times\R^{n-1}\times\R\times\R^{n-1}$ is generated by the 5 invariants
$$q_1,\quad p_1,\quad\|q_2\|^2,\quad \|p_2\|^2,\quad q_2\cdot p_2.$$

We now restrict to $T^*\Ss^{n-1}$ by requiring $\|q\|^2=q_1^2+\|q_2\|^2=1$ and $q\cdot p = q_1p_1+q_2\cdot p_2=0$.
In place of 5 invariants, we now have 3:
$$q_1,\quad p_1,\quad \text{and}\quad P:=\|p\|^2=p_1^2+\|p_2\|^2.$$

\begin{proposition} \label{prop:axisymmetric R}
The orbit space $\mathcal{R}=T^*\Ss^{n-1}/\OO(n-1)=D/G$ is a stratified semi-algebraic subspace  of $\R^3$ given by
\begin{equation*}
\mathcal{R}=\{(q_1,p_1, P)\in \R^3 \, : \, p_1^2\leq (1-q_1^2)P, \;\; -1\leq q_1 \leq 1 \, \}.
\end{equation*}
\end{proposition}

The boundary of this space is often 
called the \emph{canoe surface}.  The reduced space $\mathcal{R}$ and the dynamics on it (to be discussed below)  are  illustrated in Figure~\ref{fig:canoe}.  Basic notions of the stratification of orbits spaces are given in Appendix~\ref{app:Field}.  Further details may be found for example in~\cite{Duistermaat-Kolk}.

\begin{figure}
\centering
\begin{tikzpicture}[scale=2]
    \draw[->] (-1.5,0) -- (1.5,0) node[anchor=north east]{$q_1$};
    \draw[->] (0,0) -- (-0.6,-0.4) node[anchor=north]{$p_1$};
    \draw[fill=blue] (-1,1.5) -- (-1,0) -- (1,0) -- (1,1.5);
    \draw[fill=cyan,very thick] (0,1.5) ellipse [x radius=1, y radius=0.24];
    \draw[->] (0,1.5) -- (0,2) node[anchor=south west]{$P$};
    \draw[dashed,thick] (0,0) -- (0,1.5); 
    \draw[very thick] (0,1.5) ellipse [x radius=0.5, y radius=0.12];
    \draw[very thick] (0,1) ellipse [x radius=1, y radius=0.2];
    \draw[very thick] (0,1) ellipse [x radius=0.5, y radius=0.1];
    \draw[very thick] (0,0.5) ellipse [x radius=1, y radius=0.15];
    \draw[very thick] (0,0.5) ellipse [x radius=0.5, y radius=0.075];
    \draw (-1,0) node[anchor=north]{$-1$}; 
    \draw (1,0) node[anchor=north]{$1$}; 
\end{tikzpicture}
\caption{The reduced space $\cR$ with boundary the `canoe surface'}
\label{fig:canoe}
\end{figure}
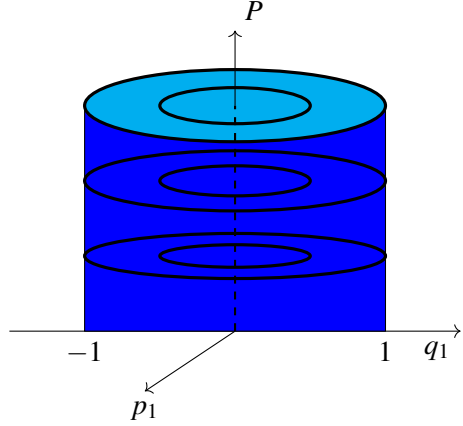

\begin{proof} It was discussed above that the quantities $q_1,p_1$ and $P$ generate the invariants.  The Cauchy-Schwarz inequality $(q_2\cdot p_2)^2\leq \|q_2\|^2\|p_2\|^2$ leads to the inequality $p_1^2 \leq P(1-q_1^2)$.
\end{proof}

Th orbit space $\cR$ has a natural stratification arising from the isotropy subgroups of the $\OO(n-1)$ action on $T^*\Ss^{n-1}$.  Accordingly, there are four strata of $\mathcal{R}$ as follows (recall that strata are by definition connected).

\begin{description} 
\item[$S_0$, $S_0'$:]  If $q_2=p_2=0$, the point $(q,p)$ is fixed by $\OO(n-1)$; this determines two  zero-dimensional strata consisting of the points $(q_1,p_1,P)=(\pm 1,0, 0)$. This corresponds to the top being placed with its axis aligned with the distinguished axis, but stationary since $P=0$; from the dynamical point of view it is no more special than any other stationary configuration of the top.
\item[$S_2$:]  This represents elements in $T^*\Ss^{n-1}$ for which $q_2$ and $p_2$ are parallel and do not both vanish.  The isotropy group of such $(q,p)$ is conjugate to $\OO(n-2)$.  The image in $\cR$ of this set is the canoe surface $p_1^2=(1-q_1^2)P$, with the points $(q_1,p_1,P)=(\pm 1,0, 0)$ removed.

\item[$S_3$:]  This 3-dimensional stratum consists of the points $(q_1,p_1,P)$ corresponding to points of $T^*\Ss^{n-1}$ for which $q_2\wedge p_2\neq 0$ and their isotropy group is conjugate to $\OO(n-3)$ (where,  for $n=4$,  $\OO(1)=\Z_2$, and for $n=3$, $\OO(0)$ is the trivial group). This stratum corresponds to the interior of the orbit space, where $p_1^2<(1-q_1^2)P$. 
\end{description}

\subsection{Reduced dynamics}
\label{sec:reduced dynamics axisymmetric}

The reduced equations on $\mathcal{R}$ can easily be obtained from \eqref{eq:RedEqnsGeneral}. One finds,
\begin{equation}
\label{eq:axisym-red-syst}
\dot q_1 =  \frac{p_1}{J_1+J_2}, \qquad
\dot p_1 = -2H q_1, \qquad
\dot P =0 , 
\end{equation}
where, after some calculations, $H$ is given by
\begin{equation} \label{eq:axisym Hred-nD}
H=\frac{1}{2(J_1+J_2)}\left ( \frac{(J_2-J_1)p_1^2+(J_1+J_2)P}{(J_1-J_2)q_1^2+2J_2}\right ),
\end{equation}
and is a first integral. As expected, the flow of these equations leaves the strata of $\mathcal{R}$ invariant.  

\paragraph{Dynamics:}  Since $H$ is a first integral, the equations of motion above are linear, with solutions
$$q_1(t) = A\cos(\omega t)+B\sin(\omega t),\qquad p_1(t)=  (J_1+J_2)\omega (-A\sin(\omega t)+B\cos(\omega t)),$$
where $\omega^2 = 2h/(J_1+J_2)$ and $A,B$ are arbitrary constants of integration and $h$ is the constant value of $H$ along the solution.  There are two classes of equilibria: along the $q_1$-axis (where $p_1=P=0$ and hence $H=0$) and along the $P$-axis (where $q_1=p_1=0$). In Figure\,\ref{fig:canoe}, the  ellipses illustrate the trajectories; they lie in horizontal planes $P=\text{const}$ and are centred on the $P$-axis.

\paragraph{Steady rotations } \label{sec:steady-rot-axisymmetric}
Following Section\,\ref{sec:Steady-rotations},  we see that for an axisymmetric top there are two kinds of principal plane of inertia:  

\begin{itemize}
\item Let $a\in\R^n$ be orthogonal to $f_1$.  Then $\{f_1,a\}$ spans a principal plane containing the axis of symmetry, an \defn{axial} principal plane, with associated moment of inertia $\inertia(\Pi)=J_1+J_2$, and the corresponding steady rotation is simply a steady rotation in this plane. 

 \item Now let $\Pi$ be any plane orthogonal to the axis of symmetry.  This is also a principal plane of inertia, called an \defn{equatorial} plane, and with $\inertia(\Pi)=2J_2$.   The resulting motion generalizes the 3D steady rotations where the body rotates about its symmetry axis.
\end{itemize}

\begin{proposition}
\label{prop:steady-rot-strata}
\begin{enumerate}  
\item The steady rotations in an axial principal plane are contained in $S_2$, and conversely all points in $S_2$ with $P>0$ correspond to such steady rotations. 
\item Steady rotations in an equatorial plane of inertia are contained in $S_3$ and have $q_1=p_1=0$, and conversely,
all points with  $q_1=p_1=0$ correspond to such steady rotations.
\item Both the axial and equatorial steady rotations are Lyapunov stable relative to $\OO(n-1)$. 
\end{enumerate}
\end{proposition}

\begin{proof}
First notice that if $\Omega=a\wedge b$  defines a steady rotation solution then, by Equation~\eqref{eq:Omqqdot}, $a\wedge b= q\wedge \dot q$ and it follows
that $q(t)$ and $\dot q(t)$  are contained in the plane spanned by $a$ and $b$ at all time. We will make use of this in the following.

(i) For a steady rotation in an axial plane we can write
$\Omega=f_1\wedge a$ for a constant vector $a$ such that $f_1\cdot a=0$. Since $q$ lies on the plane spanned by $f_1$ and $a$, we 
have $q=q_1f_1\pm\frac{ \sqrt{1-q_1^2}}{\|a\|}a$.
Hence, $$p=-\I(\Omega)q=-(J_1+J_2)\Omega q=(J_1+J_2)\left (\mp \sqrt{1-q_1^2}\,\|a\| f_1 + q_1a \right ).$$
The above formulae imply that $q_2$ and $p_2$ are parallel so the solution is indeed contained in $S_2$.

For the converse, recall that along $S_2$ the vectors $q_2$ and $p_2$ are parallel. 
Therefore, in view of \eqref{eq:RedEqnsGeneral} we conclude that $\dot q_2$ is also parallel to $q_2$. Writing $\dot  q_2=\lambda q_2$ we
obtain 
\begin{equation*}
\Omega = q\wedge \dot q = (q_1f_1 + (0,q_2))\wedge (\dot q_1f_1 + \lambda (0 ,  q_2))= f_1 \wedge a,
\end{equation*}
where $a=(\lambda q_1 -\dot q_1)\left (0, q_2 \right )$.  Due to our symmetry assumptions on $\I$, this implies that $\I(\Omega)=(J_1+J_2)\Omega$ so
$[\I\Omega, \Omega]=0$ which, by Proposition~\ref{prop:steady rotations on D}, implies that the solution is a steady rotation.
Therefore $a$ is constant and the steady rotation is of axial type.  (Note that the condition $P>0$ implies $\Omega \neq 0$).

(ii)  If $\Omega=a\wedge b$ with $a, \, b, \, f_1,$ mutually perpendicular, then, since $q$ is contained in the plane spanned by
$a$ and $b$, $q_1=q\cdot f_1=0$. But also $p_1=p\cdot f_1=0$ since
$p=-\I(\Omega)q=-2J_2(a\wedge b) q$ is also contained in the plane spanned by $a$ and $b$.

For the converse, note that these points are equilibria on $\mathcal{R}$ satisfying $q_1=\dot q_1=0$. Considering that $\Omega=q\wedge \dot q$, and our symmetry assumptions on the inertia tensor, we have $\I(\Omega)=2J_2\Omega$.  Therefore $[\I(\Omega),\Omega]=0$, and
Proposition~\ref{prop:steady rotations on D} implies that the motion is a steady rotation which 
takes place on an equatorial plane since $f_1$ is orthogonal to $q$ and $\dot q$.

(iii) This follows immediately from Theorem\,\ref{thm:stability}.
\end{proof}

\subsection{Reconstruction}

We now give the global details of the reconstruction of the dynamics both to $T^*\Ss^{n-1}$ and to $D$. The first conclusion is that the dynamics is essentially that of the axisymmetric 3D Veselova top,  in a sense that is made precise in the following theorem.

\begin{theorem}\label{thm:nD to 3D}
Consider the $n$-dimensional axisymmetric Veselova top, with $n> 3$ and mass tensor $\J$ given in \eqref{eq:axisymmetric J}.
Consider any solution to the reduced equation on $\cR$, and a choice of initial lift to both $T^*\Ss^{n-1}$ and $D$.  Then there exist choices of basis of\/ $\R^n$ in the body and in space, the first containing $f_1$ and the second $e_1$,  such that 
\begin{equation}\label{eq:axisymmetric nD to 3D}
q(t) = \begin{pmatrix}
\bar q(t)\cr 0
\end{pmatrix},\quad p(t) = \begin{pmatrix}
\bar p(t)\cr 0
\end{pmatrix},\quad
g(t) = \begin{pmatrix}
\bar g(t)& 0\cr 0&\mathrm{Id}_{n-3}
\end{pmatrix},\quad
\Omega(t) = \begin{pmatrix}
\bar \Omega(t)&0\cr 0 &0_{n-3}
\end{pmatrix}
\end{equation}
where $(\bar q(t),\bar p(t))\in T^*\Ss^2$ and $(\bar g(t), \bar\Omega(t))\in D(3)\subset T\OO(3)$ satisfy the equations of motion of the axisymmetric 3D Veselova top (for $\bar g,\bar\Omega$) and its reduced equations (for $(\bar q,\bar p)$), with mass tensor $\bar\J = \diag[J_1,J_2,J_2]$. 
\end{theorem}

\begin{proof}
Suppose that the given solution in $\cR$ is in the open stratum $S_3$ (the other cases follow similarly), and let $q_0,p_0,g_0,\Omega_0$ be corresponding initial values for the lifted dynamics.  We describe the necessary changes of basis. 
  
To begin, let $V_0$ be the subspace of $\R^n$ spanned by $q_0,p_0,f_1$ (which are necessarily linearly independent by virtue of being in the stratum $S_3$), and choose an orthonormal basis of $V_0$ consisting of $f_1$ and two further vectors.  Finally complete this to an orthonormal basis $\{f_1,\dots,f_n\}$ of $\R^n$; being orthogonal to $f_1$, the remaining vectors will all be principal directions.  The change of basis matrix $P_b\in G_R=\OO(n-1)$ maps $V_0$ to $\R^3\times \{0\}\subset\R^n$. Then $q_0, p_0$ has the form given in \eqref{eq:axisymmetric nD to 3D}.  

For the matrices, define now a basis in space as follows: let $V_0'=gV_0$. Clearly $e_1\in V_0'$ since $q_0\in V_0$.  Extend $\{e_1\}$ to an orthonormal basis of $V_0'$ and then further extend to $\R^n$ by putting $e_j=gf_j$ for $j>3$. Let $P_s\in G_L=\OO(n-1)$ be the change of basis matrix.  This maps $V_0'$ to $\R^3\times\{0\}$. 

With this choice of basis, $q_0,p_0,g_0$ are of the form given in \eqref{eq:axisymmetric nD to 3D}.  It remains to show the same is true of $\Omega$ and then that the space is fixed by the dynamics. 
Now, from Equation \eqref{eq:Omqqdot}, $\Omega=q\wedge \dot q$ and under the choice of basis transforms to $P_b\Omega P_b^T = (P_bq)\wedge(P_b \dot q)$.  Since from the equations of motion \eqref{eq:RedEqnsGeneral}, $\dot q\in V_0$ it follows that $P_b\Omega P_b^T$ is of the required form since both $P_bq$ and $P_b \dot q$ belong to $\R^3\times\{0\}$. 

We need now to show that the space of such matrices and vectors is invariant under the dynamics.  There are two ways to do this, either a symmetry argument or an inspection of the equations of motion.  For the symmetry argument, first consider the $G_R=\OO(n-1)$ action on 
$T^*\Ss^{n-1}$ (Corollary~\ref{coroll:symmetry on T^*S}) and note that $V_0$ is the fixed point set of $\OO(n-3)\subset\OO(n-1)$. It  follows that $(q(t),p(t))$ is fixed by $\OO(n-3)$ for all $t$, and hence that it evolves on $(V_0\times V_0)\cap T^*\Ss^{n-1}=T^*\Ss^2$. Secondly,
consider the action of $G=G_L\times G_R = \OO(n-1)\times \OO(n-1)$ on the pair of matrices $(g,\Omega)\in D$ (Equation~\eqref{eq:lifted action}). 
Similar to  the previous argument, the set  matrix pairs $(g,\Omega)$  given in \eqref{eq:axisymmetric nD to 3D} is the fixed point set of the diagonal subgroup 
 $$\OO(n-3)_\Delta = \{ (k,k)\in G \, | \, k\in \OO(n-3) \},$$
 and hence that too is invariant under the dynamics.

From the equations of motion \eqref{eq:VesEqns} and \eqref{eq:RedEqnsGeneral} it is easy to see that the resulting equations of motion on these fixed point sets are precisely those arising from the 3D axisymmetric Veselova top with the given mass tensor. 
\end{proof}

It follows from this theorem that in order to study the dynamics of the axisymmetric  $n$-dimensional Veselova top, it suffices to reconstruct the dynamics for the 3D one. And because of the reconstruction formula~\eqref{eq:3d-reconst-g} given in Proposition\,\ref{prop:reconstruction 3D}, each motion on $T^*\Ss^2$ lifts to a similar motion on $D$.  Recall that equatorial and axial planes are defined just prior to  Proposition~\ref{prop:steady-rot-strata}.

\begin{theorem}\label{thm:reconstruction-axisymmetric}
Consider a solution on $\cR$ of the reduced $3$-dimensional axisymmetric Veselova top, with mass tensor $\J = \diag[J_1,J_2,J_2]$. The possible reconstructions to both $T^*\Ss^2$ and $D$ are given in the following table,
\begin{center}
\begin{tabular}{r|l}
stratum & reconstruction \\
\hline
 $S_3$, with  $(q_1,p_1)\neq0$  & quasiperiodic motion on 2-tori \\
$(0,0,P)\in S_3$
& steady rotations in equatorial plane \\
 $S_2$ with $P>0$ & steady rotations in axial plane  \\
all remaining points & equilibrium
\end{tabular}
\end{center}
\end{theorem}

\begin{proof}
The proofs of the statements in the second and third rows of the table were given in Proposition~\ref{prop:steady-rot-strata}.

For the first row, the solutions  through these points in the second reduced space $\cR$ are periodic. Moreover, the action of  $G_R=\OO(2)$ on $T^*\Ss^2$ restricts to a free action on $S_3$.
Since $\OO(2)$ has  rank 1,  Field's theorem (Theorem\,\ref{thm:Field}) implies that the motion is
quasi-periodic of tori of dimension two on $T^*\Ss^{2}$.  

In the remaining case $P=0$ in which case $\Omega=0$ and the kinetic energy vanishes. Since there is no potential, there is no motion.
\end{proof}

\begin{remark}
\label{rmk:Subtori}
We note the application of Field's theorem to the reconstruction of orbits on $S_3$ to $D$ proves that the motion is quasi-periodic
on invariant 3-tori since the group $G=G_L\times G_R=\OO(2)\times \OO(2)$ has rank 2. Theorem~\ref{thm:reconstruction-axisymmetric} instead shows
that the dynamics takes place in a finer  foliation by  invariant 2-dimensional sub-tori. 
A physical interpretation of this phenomenon is given  in Section~\ref{SS:physical description} below. 
This scenario, where the invariant tori on  $T^*\Ss^{n-1}$ and on $D$ have the same dimension, also holds for general non-physical inertia 
tensors satisfying~\eqref{eq:inertia-Fed-Jov} (see \cite[Section 7]{FedJov}). 
\end{remark}

\paragraph{The energy-momentum map} The global description of the dynamics of the  $n$-dimensional axi-symmetric Veselova top may now be
conveniently illustrated by looking at the image of the energy-momentum map $(P,H)$.  We assume that  $J_1<J_2$ (the opposite inequality leads to a similar description).  
 Using that $p_1^2\leq P(1-q_1^2)$, it is straightforward to show that $P/4J_2\leq H(q,p) \leq P/2(J_1+J_2)$.
Therefore, the image of the energy-momentum map $(P,H)$ is the wedge shown in Figure~\ref{fig:E-M for axisymmetric}, see also Proposition\,\ref{prop:EM critical rays}. 

\begin{figure}
\centering
\begin{tikzpicture}[scale=2.5]
 \draw[->] (0,0) -- (1.4,0) node[anchor=north]{$P$};
 \draw[->] (0,0) -- (0,1.4) node[anchor=east]{$H$};
  \draw[fill=lightgray] (0,0) -- (0.4,1.2) -- (1.2,0.4);
  \draw[very thick] (0,0) -- (0.4,1.2) node[anchor=south west] {\hskip-5mm$P=2(J_1+J_2)H$};  
  \draw[very thick] (0,0) -- (1.2,0.4) node[anchor=west] {$P=4J_2H$};
\end{tikzpicture}
\begin{minipage}{0.9\textwidth}
\caption{Image of the energy-momentum map for the axisymmetric Veselova top, assuming $J_1<J_2$; see Figure\,\ref{fig:E-M for Veselova}} 
\label{fig:E-M for axisymmetric}
\end{minipage}
\end{figure}
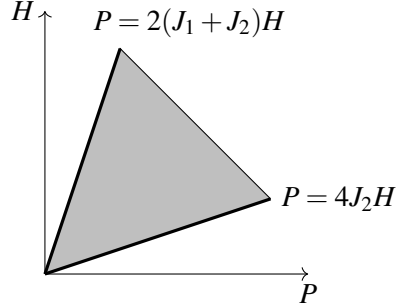

Moreover, for $(P,H)\neq(0,0)$, the equalities $H=P/2(J_1+J_2)$ and $H= P/4J_2$ are respectively attained on $S_2$ and along the points of $S_3$ where $q_1=p_1=0$ and $P>0$. 
Therefore, combining the conclusion of Theorem~\ref{thm:nD to 3D} with the  different items of Theorem~\ref{thm:reconstruction-axisymmetric}, we  obtain:
\begin{enumerate}\renewcommand{\labelenumi}{(\Alph{enumi})}
\item[(i)]  The pre-image of the points in the interior of the wedge (the regular values of the map) is foliated by invariant tori  of dimension $2$ on both $T^*\Ss^{n-1}$ and $D$ that carry quasiperiodic flow. 

\item[(ii)] The pre-image of the   lower boundary line $H=P/4J_2$, $P>0$, consists of steady rotations in equatorial principal planes of inertia.
\item[(iii)]  The pre-image of the upper boundary line $H=P/2(J_1+J_2)$, $P>0$, consists of steady rotations in axial principal planes of inertia.
\item[(iv)]   The pre-image of the vertex of the wedge consists of all the equilibria both on $D$ and on $T^*\Ss^{n-1}$.
\end{enumerate}

 \subsection{Physical description of the dynamics}
 \label{SS:physical description}

The dynamics of the  $n$-dimensional axisymmetric Veselova top is now easily described.  In view of Theorem~\ref{thm:nD to 3D}, we  restrict our attention to the 3D case.
Consider a solution of the system with $H, P\neq 0$. Without loss of generality, we may choose the  second axis $e_2$ of the
 space frame to be aligned with the space representation of the momentum vector $p(t)$. The 
 attitude matrix $g(t)$ is then given by Equation~\eqref{eq:3d-reconst-g} with $g_0=\mbox{Id}_3$.
 It follows that the axis of symmetry of the body, as seen in space, traces the following path along the motion:
 \begin{equation*}
g(t)f_1=\left (q_1(t),\frac{p_1(t)}{\sqrt{P}}, \pm \sqrt{L(t)} \right )^T,
\end{equation*}
where $L(t)=1-q_1(t)^2-\frac{p_1(t)^2}{\sqrt{P}}$. The above is  a closed curve in space since all entries of the above
vector are periodic with the same frequency. 

Let us focus on the case where  the solution is quasi-periodic on a 2-torus (see Theorem~\ref{thm:reconstruction-axisymmetric}).
Then $0<L(t)<1$ and, as follows from the conservation of energy~\eqref{eq:axisym Hred-nD}, the axis of the body rotates about the $e_3$ axis of the space frame
along the  curve given as the intersection of the cylinder
 \begin{equation*}
 \left \{  (x_1,x_2,x_3) \,   \left   |  \, x_1^2+\frac{P}{2H(J_1+J_2)}x_2^2= \frac{4HJ_2-P}{2H(J_2-J_1)} \right . \right \}
\end{equation*}
with the unit sphere. Under our assumptions, the connected components of this 
curve are not contained on a plane with constant $x_3$. Therefore, apart from the periodic precession of the body about the $e_3$-axis,
the body undergoes a periodic nutation with half its period (see Figure~\ref{F:motion-axi}). 
Apart from this periodic translational motion, the body rotates about its symmetry axis at an angular speed that 
is compatible with the nonholonomic constraint. 

\begin{figure}[h]
\centering
\includegraphics[width=3cm]{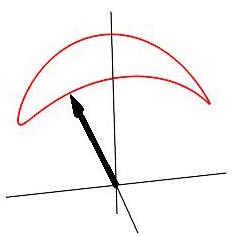}
 \put (-85,20) {$e_1$}
 \put (-32,5) {$e_2$}
  \put (-43,80) {$e_3$}
  
  \begin{minipage}{0.9\textwidth}
  \caption{Periodic motion of the axis of symmetry of the top. The distinguished axis  that defines the nonholonomic constraint is  $e_1$.
The axis   $e_2$ is aligned with the space representation of the momentum vector $p(t)$.  }
\label{F:motion-axi}
  \end{minipage}
 \end{figure}

 Note that even though the system has the same symmetry group as the classical Lagrange top, whose motions are  generically quasi-periodic with three frequencies (see e.g.\ \cite{MMCM}), the motion of the axisymmetric Veselova top only has two independent frequencies
since the nutation and precession motions of the body are commensurate. This gives a physical interpretation of Remark~\ref{rmk:Subtori}.

\section{The cylindrical Veselova top}
\label{sec:cylindrical}

We now consider the more general case of \defn{cylindrical tops}: that is, we assume 
\begin{equation}\label{eq:cylindrical J}
\J = \diag[J_1,\dots,J_1,J_2,\dots,J_2],
\end{equation}
with $J_1\neq J_2$. Let $r$ be the multiplicity of eigenvalue $J_1$, and $r'=n-r$ the multiplicity of $J_2$, and we may assume $2\leq r\leq r'$ (the case $r=1$ was treated in the previous section). In particular the dimension $n$ of the body must be at least 4. 

The symmetry group $G_R$ defined by Equation~\eqref{eq:def-G_R} is $G_R:=\OO(r)\times\OO(r')\subset\OO(n)$. Indeed, if 
$$h=\left[\begin{matrix}
h_{11}&0\cr 0&h_{22}
\end{matrix}\right],$$
with $h_{11}\in\OO(r)$ and $h_{22}\in\OO(r')$, then $h\J h^T=\J$. Corresponding to the splitting $\R^n=\R^r\times\R^{r'}$, write $q=(q_1,q_2)$ and $p=(p_1,p_2)$. Then, the action of $h\in G_R$ on $(q,p)\in T^*\Ss^{n-1}$ (see Corollary~\ref{coroll:symmetry on T^*S}) is 
$$hq=(h_{11}q_1,\,h_{22}q_2)$$ 
and similarly for $p$.

\subsection{Second reduction}
As for the axisymmetric tops, denote the \defn{second reduced space} $T^*\Ss^{n-1} /G_R$ by $\cR$. 
To determine $\cR$ for a cylindrical body, we find the invariant functions for the $G_R$ action on $T^*\Ss^{n-1}$. 
As mentioned before, the ring of invariants for $\OO(k)$ acting on $(q,p)\in\R^k\times\R^k$ is generated by three independent quadratic expressions:
$$\|q\|^2,\quad \|p\|^2,\quad\text{and}\quad q\cdot p.$$
By the Cauchy-Schwarz inequality, these satisfy $|q\cdot p|^2 \leq\|q\|^2\|p\|^2$.

It follows that the ring of invariants for $G_R=\OO(r)\times\OO(r')$ acting on $(q_1,q_2,p_1,p_2) \in \R^r\times\R^{r'}\times\R^r\times\R^{r'}$ is generated by the 6 invariants
$$\|q_1\|^2,\quad \|p_1\|^2,\quad q_1\cdot p_1,\quad\|q_2\|^2,\quad \|p_2\|^2,\quad q_2\cdot p_2.$$

We now restrict to $T^*\Ss^{n-1}$ by requiring $\|q\|^2=\|q_1\|^2+\|q_2\|^2=1$ and $q\cdot p = q_1\cdot p_1+q_2\cdot p_2=0$.
In place of 6 invariants, we now have 4:

\begin{equation}\label{eq:cylindrical invariants}
A=\|q_1\|^2,\quad B=\|p_1\|^2,\quad P=\|p_1\|^2+\|p_2\|^2,\quad\text{and}\quad D=p_1\cdot q_1.
\end{equation}
The remaining  invariants are given simply by $\|q_2\|^2=1-A$, $\|p_2\|^2=P-B$, and $q_2\cdot p_2=-D$. The reason for using $P$ rather than $\|p_2\|^2$ is that $P$ is conserved by the dynamics. (Note that $D$ refers both to this variable as well as the distribution in $T\OO(n)$---we hope no confusion will be caused.)

\begin{proposition}\label{prop:reduced space-cylindrical}
The second reduced space $\cR$ is the semialgebraic variety given by the set consisting of those $(A,B,P,D)\in\R^4$ satisfying
$$0\leq A\leq 1, \quad 0\leq B\leq P,\quad D^2\leq AB,\quad D^2 \leq(1-A)(P-B).$$
This reduced space $\cR$ has a natural stratification arising from the isotropy groups for the action of $G_R$ on $T^*\Ss^{n-1}$.  These subgroups and the associated strata, both in $T^*\Ss^{n-1}$ and their images in $\cR$, are given in Table~\ref{table:strata T^*S + R}. 
\end{proposition}

A section of $\cR$ with constant $P>0$ and the adjacencies between the strata are shown in  Figure\,\ref{fig:R-cylindrical}. 

\begin{table}[h]
\centering
$$\begin{array}{c||c|c||c}
 & \multicolumn{2}{c||}{T^*\Ss^{n-1}} & \cR \\
\hline
\text{stratum}& \text{the below are zero}& \text{ISG}& \text{equations of stratum} \\
\hline
\rule{0pt}{12pt}
S_0 & q_1, p_1, p_2 & \OO(r)\times\OO(r'-1)	 & A=B=P=D=0 \\[4pt]
S_0' & q_2, p_1, p_2 & \OO(r-1)\times\OO(r')	& A=1, B=P=D=0\\[4pt]
S_1  & q_1,p_1 & \OO(r)\times\OO(r'-2)	 & A=B=D=0  \\[4pt]
S_1' & q_2,p_2 & \OO(r-2)\times\OO(r')	 &  A=1, P-B=D=0 \\[4pt]
S_2  & \left\{\!\!\hbox{$\begin{array}{l}q_1\wedge p_1 \cr q_2\wedge p_2\end{array}$}\right. & \OO(r-1)\times\OO(r'-1) & \left\{\hbox{$\begin{aligned} D^2&=AB\cr &=(1-A)(P-B)\end{aligned}$}\right.\\[16pt]
S_3  & q_1\wedge p_1 & \OO(r-1)\times\OO(r'-2)& D^2=AB \\[4pt]
S_3' &q_2\wedge p_2 & \OO(r-2)\times\OO(r'-1) & D^2=(1-A)(P-B) \\[4pt]
S_4 &\text{--}& \OO(r-2)\times\OO(r'-2)  & \text{the rest} \cr
\hline
\end{array}
$$
\begin{minipage}{0.9\textwidth}
\caption{Strata for the action of $G_R=\OO(r)\times\OO(r')$ on $T^*\Ss^{n-1}$, valid for all   $r,r'\geq2$ with $r+r'=n$.  
Here $\OO(1)=\Z_2$ and $\OO(0)$ is the trivial group.
See Proposition\,\ref{prop:reduced space-cylindrical}.  The notation is such that $\dim S_k=\dim S_k'=k$ in $\cR$. 
}
\label{table:strata T^*S + R}
\end{minipage}
\end{table}

\begin{proof}
The inequalities arise from the definition of $A$ and the Cauchy-Schwarz inequality.

For the stratification and isotropy subgroups,  consider first the natural action of $\OO(k)$ on $x=(a,b)\in\R^k\times\R^k$, with $k\geq2$.   Clearly, the stabilizer of the point $x=(0,0)$ is $\OO(k)$, and that is the only point fixed by the whole group. Next, if $a$ and $b$ are parallel ($a\wedge b=0$) and not both zero, then the stabilizer is (conjugate to) $\OO(k-1)$ (if $k=2$ then $\OO(1)=\Z_2$ is generated by the reflection in the line containing $a,b$). Finally, if $a,b$ are linearly independent then the stabilizer is $\OO(k-2)$, which is the group of all orthogonal transformations in the subspace orthogonal to the space spanned by $a,b$.   (For the case of $k=2$, the group $\OO(0)$ is the trivial group.)

These observations can now be applied to the action of $G_R=\OO(r)\times\OO(r')$ on $\R^r\times\R^{r'}\times\R^r\times\R^{r'}$
and its restriction to $T^*\Ss^{n-1}$ to obtain the strata in Table~\ref{table:strata T^*S + R}.  
\end{proof}

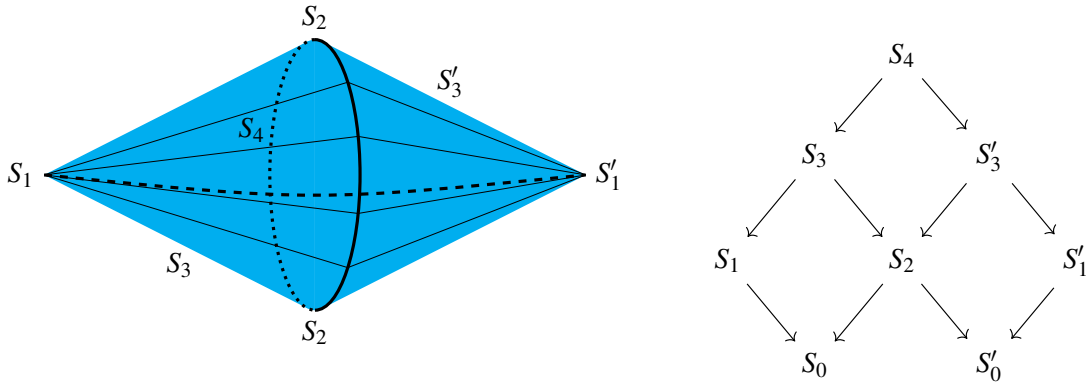
\begin{figure}[h]
\centering
\begin{tikzpicture}[scale=1.2]
\fill[cyan] (0,0) -- (2.92,1.48) -- (2.92, -1.48) -- (0,0);
\fill[cyan] (6,0) -- (3.08,1.48) -- (3.08, -1.48) -- (6,0);
\draw[very thick,fill=cyan] (3,-1.5) arc (-90:90:0.5 and 1.5); 
\draw[very thick,fill=cyan,dotted] (3,1.5) arc (90:270:0.5 and 1.5); 
\draw (0,0) node[anchor=east] {$S_1$};
\draw (6,0) node[anchor=west] {$S_1'$};
\draw (3,-1.5) node[anchor=north] {$S_2$};
\draw (3,1.5) node[anchor=south] {$S_2$};
\draw (1.5,-0.75) node[anchor=north] {$S_3$};
\draw (4.5,0.75) node[anchor=south] {$S_3'$};
\draw (2.3,0.5) node {$S_4$};
\draw[very thin] (0,0) -- (3.364, -1.028) -- (6,0);
\draw[very thin] (0,0) -- (3.479, -.4272) -- (6,0);
\draw[very thin] (0,0) -- (3.479, .4272) -- (6,0);
\draw[very thin] (0,0) -- (3.364, 1.028) -- (6,0);
\draw[very thick,dashed] (0,0) .. controls (3,-0.3) .. (6,0);
\end{tikzpicture}
\quad\begin{minipage}{0.4\textwidth}
\vskip -4cm
\[
  \begin{tikzcd}[column sep=small]
     &&S_4 \arrow[dr] \arrow[dl] &&\\
     &S_3 \arrow[dr]\arrow[dl] && S_3' \arrow[dr]\arrow[dl] \\
     S_1\arrow[dr] && S_2 \arrow[dl] \arrow[dr] && S_1'\arrow[dl] \\
     & S_0 &&S_0' 
  \end{tikzcd}
\]
\end{minipage}

\begin{minipage}{0.9\textwidth}
\caption{On the left, a section at constant $P>0$ of the reduced space $\cR$ for the cylindrical Veselova top; the $S_j$ label the strata; it is a double solid cone, meeting at the `equator' which is the stratum $S_2$. The dashed curve shows the (relative) equilibria in the open stratum $S_4$ (the interior of the cone).  As $P$ decreases, the cone becomes `thinner' and in the limit when $P=0$ it becomes a line segment (the $A$-axis in $\cR$).    On the right is shown the adjacencies between the strata, where $X\longrightarrow Y$ means $Y\subset \overline{X}$. }
\label{fig:R-cylindrical}
\end{minipage}
\end{figure}

\paragraph{Description of strata} Here we give a brief description of the strata, in particular in the constraint distribution $D$; the statements about dynamics are proved in Proposition~\ref{prop:steady-rot-strata-Cylindrical} and Theorem\,\ref{thm:reconstruction-cylindrical}.
\begin{description}
\item[$S_0,S_0'$:] These points in $\cR$ correspond to points in $T^*\Ss^{n-1}$ where $p=0$, and either $q_1=0$ or $q_2=0$ respectively. This means the top is oriented in space in such a way that the distinguished axis is contained in one of the eigenspaces of $\J$, and the top is stationary.
These strata do not appear in the cone in Figure\,\ref{fig:R-cylindrical} since the cone represents  a section of $\cR$ with $P>0$.
\item[$S_1,S_1'$:] Here the top is such that again the distinguished axis lies in one of the eigenspaces of $\J$; we will see below that the corresponding motion is necessarily a steady rotation (of {\em pure type}; see below). These strata  correspond to the
vertices of the cones in  Figure\,\ref{fig:R-cylindrical}.
\item[$S_2$:] The configuration is arbitrary, but we show below the motion is a steady rotation, here of {\em mixed type} (see below).  In Figure\,\ref{fig:R-cylindrical} the stratum $S_2$ is the circle where the two cones meet.
\item[$S_3,S_3'$:] In Figure\,\ref{fig:R-cylindrical} these strata are the boundaries of the two solid cones.  We show below that in fact the motion on these strata is `isomorphic' to the motion of an axisymmetric top. 
\item[$S_4$:] In the figure, this corresponds to the interior of the solid cone. This stratum contains the `genuine' general motion of the cylindrical top, not reducible to other cases. For generic points, the solutions are quasiperiodic on tori of dimension 4;  there are also some relative equilibria where the tori are of dimension  3, represented by the dashed curve in the figure. 
\end{description}

\begin{proposition}\label{prop:reducedEqMotion}
The reduced equations of motion for the cylindrical Veselova top on $\cR$ are,
\begin{equation} \label{eq:reduced EoM}
\left\{\quad\begin{aligned}
\dot A \  &= \ \frac{2D}{J_1+J_2}, \\ 
\dot B \ &= \ -4HD, \\
\dot P \ &= \ 0, \\
\dot D \ &= \ -4HA + \frac{P-4J_2 H}{J_1-J_2}. 
\end{aligned}\right.
\end{equation}
Here 
\begin{equation}\label{eq:Hamiltonian in ABPD}
H = \frac12\left((\beta_1-\beta_2)B+\beta_2P - \frac{(\beta_1-\beta_2)^2D^2}{\beta_1A+\beta_2(1-A)}\right),
\end{equation}
which is an integral of motion, and 
\begin{equation}\label{eq:betas}
\beta_1=((J_1-J_2)A+J_1+J_2)^{-1}, \quad\quad \beta_2=( (J_1-J_2)A+2J_2)^{-1}.
\end{equation}
\end{proposition}

\begin{proof}  We derive these from the equations of motion given in Proposition\,\ref{prop:eqns-nD}.
Under our assumption on $\J$ we have 
$$q\cdot\J q = J_1\|q_1\|^2 +J_2\|q_2\|^2 = J_1A+J_2(1-A).$$
It follows that the matrix $C$ of \eqref{eq:Cmatrix} satisfies
$$C^{-1} = \begin{pmatrix}
\left(J_1(1+A)+J_2(1-A)\right)\mathrm{Id}_r &0\cr 0 & \left(J_1A+J_2(2-A)\right)\mathrm{Id}_{r'}
\end{pmatrix} 
$$
Then $C=\begin{pmatrix}
\beta_1\mathrm{Id}_r &0\cr 0&\beta_2\,\mathrm{Id}_{r'} 
\end{pmatrix}$
where $\beta_1,\beta_2$ are given in \eqref{eq:betas}.  
Using this notation, we find
\begin{eqnarray*}
    q\cdot Cq &=& \beta_1A+\beta_2(1-A) \\
    p\cdot Cp &=& \beta_1B+\beta_2C \\
    q\cdot Cp &=& (\beta_1-\beta_2)D,
\end{eqnarray*}
and the Hamiltonian \eqref{eq:energy-general} is indeed given by \eqref{eq:Hamiltonian in ABPD}.

Now,
$$\dot A = 2q_1\cdot\dot q_1 = 2q_1\cdot\left(\beta_1p_1 - \frac{(\beta_1-\beta_2)D}{\beta_1A+\beta_2(1-A)}\beta_1q_1\right) = 2\beta_1D\left(1-\frac{(\beta_1-\beta_2)A}{\beta_1A+\beta_2(1-A)}\right).$$
and after some simple calculations one finds the expression for $\dot A$ given above. 

More simple is, 
$$\dot B = 2p_1\cdot\dot p_1=-4HD,\quad \dot P = 0,$$
and
$$\dot  D = -2HA+\beta_1B - \frac{\beta_1(\beta_1-\beta_2)D^2}{\beta_1A+\beta_2(1-A)},$$
which can be shown to be equal to the required expression.
\end{proof}

\begin{remark}\label{rmk:reduced measure}
The invariant measure for the general (first reduced) $n$-dimensional Veslova top given in \eqref{prop:measure-nD}  pushes forward to an invariant measure on $\cR$ given simply by
$$ \beta_1\beta_2\,\d A\,\d B\,\d P\,\d D.$$
\end{remark}

\paragraph{Reduced dynamics}
We can now describe the reduced dynamics on $\cR$. 
Since $P,H$ are constant, the equations of motion \eqref{eq:reduced EoM} are linear, and easy to integrate. First consider the equations for $(A,D)$: the general solution is
\begin{equation}\label{eq:AD solutions}
\begin{split}
A(t) &= C_1\cos(\omega t) + C_2\sin(\omega t) + A_* \\
D(t) &= \frac12(J_1+J_2)\omega \left(-C_1\sin(\omega t) + C_2\cos(\omega t)\right),
\end{split}
\end{equation}
where $C_1,C_2$ and $A_*$ are arbitrary constants (if the constants are chosen so that the initial value lies within $\cR$, then the entire solution must lie within $\cR$).  Here
\begin{equation}\label{eq:omega}
\omega^2 = \frac{8h}{J_1+J_2},
\end{equation}
where $h$ is the constant value of $H$ on the solution.    Substituting into the differential equation for $B$ shows,
\begin{equation}\label{eq:B solution}
B(t) = -2\,h (J_1+J_2) \left(
C_1\,\cos (\omega t) + C_2\,\sin(\omega t) \right) + B_*. 
\end{equation}
The quantities $A_*$ and $B_*$ (depending on the initial conditions) are the respective mean values of $A$ and $B$ over one period.  The mean value for $D$ is $D_*=0$.

\paragraph{Equilibria}
From $\dot A=0$ we require $D=0$ which then implies $\dot B=0$.  There remains to consider the equation $\dot D=0$.
\begin{itemize}
\item  As already noted above, along the $A$-axis, the points are equilibria (these correspond to $p=0$ so are the equilibria of the full system, and are the points where $H=0$).  The end points of this axis are the points on $S_0$ and $S_0'$, while the other points belong to $S_2$.
\item The strata $S_0,S_0',S_1,S_1'$ consist entirely of equilibrium points (for example, on $S_1$, if $A=B=D=0$ then $P=4J_2 H$ showing $\dot D=0$). 

\item More generally, equilibria occur with $H=h>0$, if $A=A_0$ and $B=B_0$, where
\begin{equation}\label{eq:equilibria}
A_0 = \frac{P-4J_2h}{4(J_1-J_2)h},\quad\text{and}\quad 
B_0 = \frac{(P-4J_2h)(P+4J_1h)}{8(J_1-J_2)h}.
\end{equation}
If $P=4J_2h$ then $A_0=B_0=0$ and the point lies in $S_1$; if $P=4J_1h$ then $A=1,B=P$ and the point lies in $S_1'$.  All other equilibrium points lie in $S_4$, and are illustrated by the dashed curve in $\cR$ in Figure\,\ref{fig:R-cylindrical}.
\end{itemize}

\begin{theorem} \label{thm:dynamics on R}
The motion on the reduced space $\cR$ has the following properties.
\begin{enumerate}
\item It is integrable, with the 3 first integrals $P,H$ and 
$$F:= 2(J_1+J_2)AH + B.$$
\item Each point with $H=0$ is an equilibrium. 
\item For each $(h,P)$  with $h>0$ in the image of the energy-momentum map (see Figure\,\ref{fig:E-M for cylindrical}) there is a unique equilibrium with $H=h$; this occurs at the point $(A,B,P,D) = (A_0,B_0,P,0)$, where $A_0$ and $B_0$ are the expressions given in \eqref{eq:equilibria}. 
\item All other solutions are periodic with period $\pi\sqrt{\frac{J_1+J_2}{2h}}$.  
\end{enumerate}
\end{theorem}

In particular the period of the periodic orbits given in (iv) depends only on the value of the energy; this is reminiscent of the period-energy relation in Hamiltonian systems, even though this system is not known to be Hamiltonian (not even Poisson-Hamilton systems in general satisfy this property).

\begin{proof}
(i) For $P$ this is clear, for $H$ we know this already, while for $F$ it is a simple calculation:
\begin{eqnarray*}
\dot F &=& 2(J_1+J_2)\dot A H + \dot B \\
&=& 2H\left((J_1+J_2)\dot A -2D\right)\ =\ 0.
\end{eqnarray*}

\noindent(ii) The set of points with $H=0$ is the $A$-axis in $\cR$, and we have already seen above that these are equilibria.

\noindent(iii) Solving the equation $\dot D=0$ with $H=h$ for $A$ shows immediately that $A=A_0$. Putting $A=A_0$ the equation $H=h$ can easily be solved for $B$ to find $B=B_0$.

\noindent(iv) If a point is not an equilibrium then it lies on a periodic orbit, as shown in \eqref{eq:AD solutions} and \eqref{eq:B solution}, with period $2\pi/\omega$ with $\omega$ given in \eqref{eq:omega}.
\end{proof}

\paragraph{Conclusion:} In $\cR$ each invariant set given by fixing values of $H$ and $P$  is a surface bounded by the intersection with a circle in the boundary of the cone in Figure\,\ref{fig:R-cylindrical}, with a unique equilibrium point surrounded by periodic orbits all of the same period.

\begin{remark}\label{rmk:Poisson structure}
Since the dynamics on $\cR$ is entirely periodic, one may find Poisson structures of rank 2 for which the system is Poisson-Hamiltonian, with Hamiltonian $H$, and with Casimir functions $P$ and $F$ \cite{FassoRank2}. It would be interesting  to know if
any of these Poisson structures arises as the reduction of an almost Poisson structure
on $D$ as described in \cite{LGN-JM17}.  This possibility to express the dynamics on  $\cR$ in Hamiltonian form contrasts with the dynamics on $T^*\Ss^{n-1}$, c.f.\  Remark\,\ref{rmk:hamiltonization}.
\end{remark}

\paragraph{Steady rotations} 
It follows from Proposition\,\ref{prop:EM critical rays} that there are three critical rays in the image of the momentum map; namely where $P=4J_1H$, $P=2(J_1+J_2)H$ and $P=4J_2H$. The image of the energy-momentum map is therefore as shown in Figure\,\ref{fig:E-M for cylindrical}.

\begin{figure}[h]
\centering
\begin{tikzpicture}[scale=2.5]
 \draw[->] (0,0) -- (1.4,0) node[anchor=north]{$P$};
 \draw[->] (0,0) -- (0,1.4) node[anchor=east]{$H$};
  \draw[fill=lightgray] (0,0) -- (0.4,1.2) -- (1.2,0.4);
  \draw[very thick] (0,0) -- (0.4,1.2) node[anchor=south west] {\!\!\!$P=4J_1H$};  
  \draw[very thick] (0,0) -- (0.8,0.8) node[anchor=south west] {$P=2(J_1+J_2)H$};  
  \draw[very thick] (0,0) -- (1.2,0.4) node[anchor=west] {$P=4J_2H$};
\end{tikzpicture}
\begin{minipage}{0.9\textwidth}
\caption{Image and critical rays (see Proposition \ref{prop:EM critical rays}) of the energy-momentum map for the cylindrical Veselova top, assuming $J_1<J_2$} 
\label{fig:E-M for cylindrical}
\end{minipage}
\end{figure}
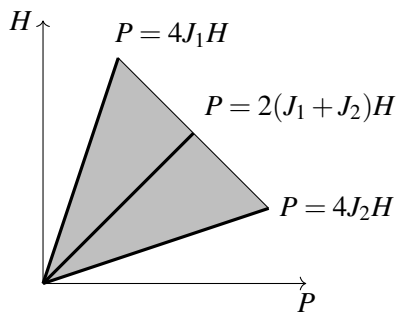

Each critical ray  in the image of the momentum map corresponds to a type of steady rotation of the 
 cylindrical Veselova top:
\begin{itemize}
\item Any plane $\Pi$ contained in the $r$-dimensional $J_1$-eigenspace of $\J$ is a principal plane of inertia.  Such a plane $\Pi$ satisfies $\inertia(\Pi)=2J_1$.
\item Any plane $\Pi$ contained in the $r'$-dimensional $J_2$-eigenspace of $\J$ is a principal plane of inertia.  Such a plane $\Pi$ satisfies $\inertia(\Pi)=2J_2$. 
\item Let $\J a =J_1a$ and $\J b=J_2b$, with $a,b$ non-zero.  The plane $\Pi_{a,b}=\R\{a,b\}$ is a   principal plane in the body with moment of inertia $\inertia(\Pi)=J_1+J_2$.
\end{itemize}
The inertial planes of the first two types we call \defn{pure inertial planes}, while those of the third type we call  \defn{mixed inertial planes}.  The steady rotations in pure inertial planes are relative equilibria, while the ones in mixed inertial planes are not (see Theorem\,\ref{thm:reconstruction-cylindrical}).

Analogous to Proposition~\ref{prop:steady-rot-strata}, we have:
\begin{proposition}
\label{prop:steady-rot-strata-Cylindrical}
\begin{enumerate}  
\item Steady rotations of mixed type are contained in $S_2$, and conversely,
all points in $S_2$  with $P>0$ correspond to such steady rotations.
\item The steady rotations in a  pure inertia plane associated to $J_1$ are contained in $S_1'$, and conversely all points in $S_1'$ correspond to such steady rotations. The same statement holds replacing  $J_1$  by $J_2$ and $S_1'$ by $S_1$.
\item The pure steady rotations are Lyapunov stable relative to the action of $G_R$. 
\end{enumerate}
\end{proposition}
\begin{proof}
The proof is a straightforward adaptation of the proof of Proposition~\ref{prop:steady-rot-strata}. For example, let us show that all points in $S_2$
with $P>0$ 
correspond to steady rotations of mixed type. Recall from Table~\ref{table:strata T^*S + R} that along this stratum  $q_j$ is parallel to $p_j$,  $j=1,2$.
In view of the cylindrical symmetry of $\J$,  Equations~\eqref{eq:RedEqnsGeneral} 
then imply that $\dot q_j$ is parallel to $q_j$. Writing $\dot q_j=\lambda_j q_j$ we obtain
\begin{equation*}
\Omega = q\wedge \dot q = ((q_1,0)+ (0,q_2))\wedge (\lambda_1(q_1,0) + \lambda_2 (0 ,  q_2))=(\lambda_2-\lambda_1) (q_1,0)\wedge (0,q_2).
\end{equation*}
The cylindrical symmetry of  $\J$ now implies that $\I(\Omega)=(J_1+J_2)\Omega$ so
$[\I\Omega, \Omega]=0$ which, by Proposition~\ref{prop:steady rotations on D}, implies that the solution is a steady rotation, that is clearly of mixed type.
(Note that the condition $P>0$ implies $\Omega \neq 0$.)

The details of the proof of the other statements are left to the reader.  The final statement follows from Theorem\,\ref{thm:stability}.
\end{proof}

\subsection{Reconstruction to $T^*\Ss^{n-1}$ and $D$} \label{sec:lifting - cylindrical}

Recall that a Veselova top with cylindrical symmetry is necessarily of dimension at least 4.  The analogous reduction to 3D for the axisymmetric top in this case is a reduction to 4D. 

\begin{theorem}\label{thm:nD to 4D}
Consider the $n$-dimensional cylindrical Veselova top, with $n>4$ and mass tensor $\J$ given in \eqref{eq:cylindrical J}.
Consider any solution to the reduced equation on $\cR$, and a choice of initial lift to both $T^*\Ss^{n-1}$ and $D$.  Then there exists choices of basis of\/ $\R^n$ in the body and in space,  such that 
\begin{equation}\label{eq:cylindrical nD to 4D}
q(t) = \begin{pmatrix}
\bar q(t)\cr 0
\end{pmatrix},\quad p(t) = \begin{pmatrix}
\bar p(t)\cr 0
\end{pmatrix},\quad
g(t) = \begin{pmatrix}
\bar g(t)& 0\cr 0&\mathrm{Id}_{n-4}
\end{pmatrix},\quad
\Omega(t) = \begin{pmatrix}
\bar \Omega(t)&0\cr 0 &0_{n-4}
\end{pmatrix},
\end{equation}
where $(\bar q(t),\bar p(t))\in T^*\Ss^3$ and $(\bar g(t), \bar\Omega(t))\in D(4)\subset T\OO(4)$ satisfy the equations of motion of the axisymmetric 4D Veselova top (for $(\bar g,\bar\Omega)$) and its reduced equations (for $(\bar q,\bar p)$), with mass tensor $\bar\J = \diag[J_1,J_1,J_2,J_2]$. 
\end{theorem}

\begin{proof}
This is similar to the proof of Theorem\,\ref{thm:nD to 3D}, and we highlight the differences.  Here we consider a point in the stratum $S_4$. Let $V=V_1\oplus V_2$ be the 4-dimensional subspace of $\R^n$, where
$$V_1 \ = \ \R\,\left\{\begin{pmatrix}q_1\cr0\end{pmatrix},\;\begin{pmatrix}p_1\cr0\end{pmatrix}\;\,\right\},\quad\text{and}\quad V_2 \ = \ \R\,\left\{\begin{pmatrix}0\\ q_2\end{pmatrix},\; \begin{pmatrix}0\cr p_2\end{pmatrix} \,\right\}.$$
The assumption that the initial point is in $S_4$ implies these 4 vectors are linearly independent. 
 We now construct an orthonormal body basis $\{f_1,\dots, f_n\}$ as follows. 
First choose $\{f_1,f_2\}$ and $\{f_3,f_4\}$ as orthonormal bases of $V_1$ and $V_2$ respectively.
Next choose $\{f_5,\dots, f_{r+2}\}$ to be an orthonormal basis of the complement of $V_1$ in $\R^r\times \{0\}$
and the remaining vectors $\{f_{r+3},\dots, f_{n}\}$ to be a basis of the complement of $V_2$ in $\{0\}\times \R^{r'}$.

With this choice of body frame, $q_0$ and $p_0$ are of the form given in  \eqref{eq:cylindrical nD to 4D}. However,
contrary to the situation encountered in the proof of  Theorem\,\ref{thm:nD to 3D}, the corresponding change of basis matrix $P_b\notin G_R$.
With the new choice of body frame, the  mass tensor $\J$ transforms to 
\begin{equation*}
\J'= \diag[J_1,J_1, J_2, J_2, \underbrace{J_1,\dots,J_1}_{r-2},\underbrace{J_2,\dots,J_2}_{r'-2}].
\end{equation*}

Now, for the space frame, 
 $e_1=g_0q$ so is contained in the image $g_0(V)$, and complete $\{e_1\}$ to an orthonormal basis of $g_0(V)$. Finally
  let $e_j=g_0 f_j$ for $j>4$.  The change of basis matrix $P_s\in G_L$ and  the matrices $g_0$ and $\Omega_0$ are in the forms given in \eqref{eq:cylindrical nD to 4D}.  

To conclude invariance of the set   \eqref{eq:cylindrical nD to 4D} under the dynamics we apply a symmetry argument as in  the proof of Theorem\,\ref{thm:nD to 3D}. 
This time we consider the group $H$ with elements:
\begin{equation*}
h= \left[\begin{array}{ccc}
\mbox{Id}_4& 0& 0 \\
0 & h_{11} & 0\\
0 & 0 &h_{22} \end{array}\right], \qquad h_{11} \in \OO(r-2), \quad h_{22}\in \OO(r'-2).
\end{equation*}
Note that for  $h\in H$ we clearly have $h^T\J'h=\J'$ so  $H \subset G_R$. Similar to the argument given in
   the proof of Theorem\,\ref{thm:nD to 3D}, 
all vectors $(q,p)$ in~\eqref{eq:cylindrical nD to 4D} 
are fixed by the $H$ action on $T^*\Ss^{n-1}$ and the matrices   $(g,\Omega)$
are fixed  by the action of the  diagonal subgroup $H_\Delta$ on $D$ where
\begin{equation*}
H_\Delta = \{ (h,h)\in G= G_L\times G_R \, | \, h\in H\}.
\end{equation*}

It can readily be seen that the equations of motion on these fixed point subspaces is the same as those of the 4D cylindrical Veselova top.
\end{proof}

We now describe the dynamics of the cylindrical Veselova top in $D$, as well as its first reduction in $T^*\Ss^{n-1}$.  Since in $\cR$ the motion is entirely periodic, the lifted motions are known as \emph{relative periodic orbits}.  We take advantage of the previous theorem and assume, without loss of generality, that $n=4$.

\begin{theorem}\label{thm:reconstruction-cylindrical}
The reconstruction to $T^*\Ss^3$ and to $D$ of the solutions of the reduced equations are as follows. 
\begin{enumerate}
\item In $S_0,S_0'$ the solutions are equilibria. 
\item In $S_1,S_1'$, the relative equilibria are steady rotations in pure inertial planes (for $S_1$ they  rotate in a principal plane $\Pi$ with $\inertia(\Pi)=2J_2$ while in $S_1'$ it has $\inertia(\Pi)=2J_1$). 
\item  In $S_2$ we have two following possibilities.  If $P=0$ the top is in equilibrium, whereas
if $P>0$, the solutions are steady rotations in mixed inertial planes, and they are not relative equilibria.  
\item In $S_3,S_3'$ the motion lies on tori of dimension at most 2 in both $T^*\Ss^3$ and $D$ and behaves in the same way as the 3D axisymmetric Veselova top.
\item In $S_4$ we have two possibilities.    In the (relative) equilibria in $S_4$  the motion lies on tori of dimension at most 2 in $T^*\Ss^3$ and at most 3 in $D$.  These relative equilibria are $G_R$-Lyapunov stable. Finally, over the periodic orbits in $S_4$ the motion lies on tori of dimension at most 3 in $T^*\Ss ^3$ and at most 4 in $D$.
\end{enumerate}
\end{theorem}

\begin{proof}
(i) $S_0$ and $S_0'$ have $p=0$ and hence are equilibria. 

\noindent(ii) and (iii):  the proofs are entirely analogous to the proof of Proposition~\ref{prop:steady-rot-strata} (the case where $P=0$ in (iii) is not, but is immediate).

\noindent(iv) Suppose $(q,p)\in S_3$ ($S_3'$ is similar). Then their projections $q_1$ and $p_1$ are parallel, say to a vector $u$, while $q_2$ and $p_2$ are linearly independent. The line $\R u$ is the fixed point set of the isotropy group $\OO(1)=\Z_2$, and hence the dynamics will take place on the 3-dimensional subspace spanned by $u,q_2$ and $p_2$.  
The resulting equations of motion are precisely those of the 3-dimensional axisymmetric top with $\J=\diag[J_1,J_2,J_2]$, and the conclusions about the invariant tori are to be found in Theorem\,\ref{thm:reconstruction-axisymmetric}.

\noindent(v)    
Denote by $\widehat{S}_4$ the stratum in $T^*\Ss^3$ corresponding to the stratum $S_4$ in $\cR$. 
 It follows from Table~\ref{table:strata T^*S + R} that the $G_R=\OO(2)\times\OO(2)$ action on $\widehat{S}_4$ is free. 
 The conclusion about the dynamics on $T^*\Ss^3$ then follows from Field's theorem (Theorem\,\ref{thm:Field}) since
 $\OO(2)\times\OO(2)$ has rank 2. Similarly, if  $\widehat{\widehat{S}}_4$ denotes the corresponding stratum on  $D$,
 then the action of $G=G_L\times G_R=\OO(3)\times\OO(2)\times \OO(2)$ on $\widehat{\widehat{S}}_4$ is free 
 and the conclusion follows from the same theorem (since $G$ has rank 3).   
 The stability statement follows from the Lyapunov stability of the equilibria on $\cR$.
\end{proof}

Note that the above theorem only  gives upper bounds for the dimension of the invariant
 tori corresponding to the reconstruction of the solutions on the stratum $S_4$. We will now present some evidence to show that these bounds are 
 generically attained and the dynamics does not take place in lower
 dimensional sub-tori (compare with Remark~\ref{rmk:Subtori}). We begin with the following proposition that considers the reconstruction 
of the non-trivial relative equilibria in $S_4$.

\begin{proposition} Let $(h,P)$ with $h>0$ be a point in the  image of the energy-momentum map (see Figure\,\ref{fig:E-M for cylindrical}) 
and consider the corresponding relative equilibrium 
described by item (iii) of Theorem~\ref{thm:dynamics on R}. A reconstruction of this solution to $T^*\Ss^3$ and to $D$ is
respectively determined by:
\begin{equation*}
\begin{split}
q(t)=\left ( \sqrt{A_0}\cos \omega_1 t, - \sqrt{A_0}\sin \omega_1 t,  \sqrt{1-A_0}\cos \omega_2 t,  -\sqrt{1-A_0}\sin \omega_2 t \right )^T, \\
p(t)=\left ( -\sqrt{B_0}\sin \omega_1 t,  -\sqrt{B_0}\cos \omega_1 t,  - \sqrt{P-B_0}\sin \omega_2 t, - \sqrt{P-B_0}\cos \omega_2 t \right )^T,
\end{split}
\end{equation*}
and
\begin{equation*}
g(t)= \left[\begin{array}{cccc}
1&0 & 0 & 0\\
 0 &1& 0 & 0\\
0 & 0 & \cos \omega_3 t & -\sin \omega_3 t \\
0 & 0 &\sin \omega_3 t  & \cos \omega_3 t
\end{array}\right] \, g_0 \, \left[\begin{array}{cccc}
 \cos \omega_1 t &- \sin \omega_1 t & 0 & 0\\
 \sin \omega_1 t  & \cos \omega_1 t& 0 & 0\\
0 & 0 & \cos \omega_2 t & -\sin \omega_2 t \\
0 & 0 &\sin \omega_2 t  & \cos \omega_2 t
\end{array}\right],
\end{equation*}
where
\begin{equation*}
g_0=\left[\begin{array}{cccc}
\sqrt{A_0}&0 & \sqrt{1-A_0} & 0\\
 0 &\sqrt{B_0/P}& 0 & \sqrt{(P-B_0)/P} \\
\sqrt{1-A_0} & 0 & -\sqrt{A_0} & 0 \\
0 &  \sqrt{(P-B_0)/P} &0  & -\sqrt{B_0/P}
\end{array}\right].
\end{equation*}
The frequencies in the above formulae are given by
\begin{equation}
\label{eq:Reconstruction frequencies}
\omega_1=\frac{2\sqrt{2}h}{\sqrt{4J_1h+P}}, \qquad \omega_2=\frac{2\sqrt{2}h}{\sqrt{4J_2h+P}}, \qquad 
\omega_3=\frac{-4h\sqrt{P}}{\sqrt{4J_1h+P}\sqrt{4J_2h+P}}.
\end{equation}
\end{proposition}

The proof  for $T^*\Ss^3$ is a direct verification that $(q(t),p(t))$ is a solution of Equations~\eqref{prop:eqns-nD}. For
$D$, one may compute $\Omega=g^{-1}\dot g$ and verify
by a direct calculation  that the nonholonomic constraints and  relations such as~\eqref{eq:Legendre} 
hold. For generic $(h,P)$ the reconstruction frequencies  \eqref{eq:Reconstruction frequencies} are  non-resonant
so these relative equilibria indeed reconstruct to motions on  invariant 3-tori on $D$ and on an invariant 2-tori on $T^*\Ss^3$.

Now consider the motion along a relative periodic orbit on $\cR$ that is close to the relative
equilibrium treated above. Both the reconstructed solutions on $T^*\Ss^3$ and on $D$ are quasi-periodic.
They will posses the natural frequency given by Equation~\eqref{eq:omega} (that only depends on $h$)
 and reconstruction frequencies
that are close to those given in~\eqref{eq:Reconstruction frequencies}. The joint set of these frequencies
will generically be non-resonant and the motion will take place on a 4-torus on $D$ and on a 3-torus on $T^*\Ss^3$.

\section{Conclusions and open question}
\label{S:Question-Hamiltonisation}
Our study shows that the dynamics of the  axisymmetric and cylindrical Veselova tops is quasiperiodic in the natural time variable. 
More precisely, at the level of the first reduced space $D/G_L=T^*\Ss^{n-1}$ the flow possesses an invariant measure
and:
\begin{enumerate}
\item for the axisymmetric case  (we may take $n=3$), 
the generic motion is  quasi-periodic  on invariant $2$-tori on $T^*\Ss^2$ and there are two integrals of motion $H$ and $P$;
\item for the cylindrical case (we may take $n=4$), the generic motion is 
quasi-periodic on invariant $3$-tori on $T^*\Ss^3$ and there are three integrals of motion $H$, $P$ and $F$.
\end{enumerate}
These scenarios are remarkably similar to the situation found in Liouville-Arnold integrable Hamiltonian systems.
We leave it as an open problem to find  a Hamiltonian structure for these systems on these cotangent bundles
(one that does not involve a time reparametrisation) or to identify an obstruction to its existence.
The interest in this  question arises in view of the great quantity of research that in  recent years has been devoted to identifying the differences between nonholonomic and Hamiltonian 
systems.

\appendix

\section{Field's theorem on reconstruction}
\label{app:Field}

We briefly recall the language and main properties of (compact) group actions and equivariant vector fields, needed for the proofs of quasiperiodicity. 

Suppose a (compact Lie) group $G$ acts smoothly on a manifold $M$. If two points  lie in the same orbit, their isotropy subgroups  are conjugate. This motivates the partition or stratification of $M$ into ``orbit types,'' where the orbit type stratum $M_{(H)}$ consists of those points with isotropy subgroup (ISG) conjugate to $H$, or more precisely a stratum is a connected component of such points. 
Moreover, if $M_H$ is (a connected component of) the set of points with ISG equal to $H$, then, up to connected components, $M_{(H)} = G.M_H$ (the image of $M_H$ under the group action).  The orbit space $M/G$ is similarly stratified by the orbit type; one denotes its corresponding stratum by $(M/G)_H$.  All of these spaces $M_{(H)}, M_H$ and $(M/G)_H$ are manifolds.

If in addition there is an equivariant dynamical system on $M$, then it descends to a dynamical system on $M/G$ which respects each stratum. The question of reconstruction is, if one knows some properties of the dynamics on $M/G$ what does this imply about the dynamics on $M$. We use the following simplified version of an important theorem due to Field \cite{Field80} (see also \cite[Chapter 8]{Field-book} or  \cite{CDS-book}).

\begin{theorem}[Field \cite{Field80}] \label{thm:Field}
Consider a smooth equivariant dynamical system, with symmetry group $G$ acting freely on the phase space $M$. The dynamics passes down to the smooth orbit space $M/G$.  
\begin{enumerate}
\item  Let $x\in M/G$ be an equilibrium point of the reduced equations.
 Then the inverse image of $x$ in $M$ is a group orbit which is foliated by invariant tori of dimension at most the rank of $G$.
\item Let $\gamma$ be a periodic orbit of the reduced dynamics in $M/G$. Then the inverse image of this curve is also foliated by invariant tori, but now of dimension at most $\mathrm{rk}(G)+1$.
\end{enumerate}
In both cases the dynamics in the invariant tori is quasiperiodic.
\end{theorem}

In case (i) the dynamics on $M$ is called a \emph{relative equilibrium}, while in (ii) it is a a \emph{relative periodic orbit}. 
If the dynamics on $M/G$ include an invariant quasiperiodic torus, then it is unknown except in special cases, see \cite{FGNG},  what the corresponding reconstructed dynamics may be.  The more general version of Field's theorem does not require the action to be free, but this suffices for our purposes.

\section{Physical  inertia tensors satisfying the hypothesis of Fedorov-Jovanovi\'c}
\label{A:Inertia}

Recall that  Fedorov and Jovanovi\'c \cite{FedJov,FedJov2} work under the assumption that there is an $n\times n$ diagonal 
matrix $A$ such that the inertia tensor $\I:\so(n)\to\so(n)$ satisfies 
\begin{equation}\label{eq:inertia-tensor}
\I(a\wedge b) = (Aa)\wedge(Ab),\quad (\forall a,b\in\R^n).
\end{equation}
In this appendix we
examine the feasibility of this condition within the family of physical inertia tensors. The following proposition, that was already  suggested in \cite{FedJov2, Jovan}, shows that for 
  $n\geq 4$ a physical inertia tensor that satisfies  \eqref{eq:inertia-tensor} necessarily corresponds to an axisymmetric rigid body.
  \begin{proposition}
\label{P:phys-inertia}
Let $\I$ be a physical inertia tensor defined by $\I(\Omega)=\J\Omega+\Omega\J$ (see \eqref{eq:Phys-Inertia-Intro}).
\begin{enumerate}
\item If $n=3$ then the identity \eqref{eq:inertia-tensor} holds for arbitrary $a, b\in \R^3$ with 
$$A_i=\sqrt{\frac{(J_i+J_j)(J_i+J_k)}{J_j+J_k}},$$
for $\{i,j,k\}=\{1,2,3\}$, where $\mathbb{J}=\diag[J_1,J_2,J_3]$ (with $J_i+J_j>0$ for $i\neq j$).
\item If $n\geq 4$, there exists a diagonal matrix $A$ satisfying  \eqref{eq:inertia-tensor} if and only if\/ $\I$ is the inertia tensor of an axisymmetric body.
In this case, the body frame can be chosen in such way that $\J$ is given by  \begin{equation}
\label{eq:Lagrange-J}
\J=\mbox{diag}\left[A_1A_2-\tfrac12 A_2^2,\tfrac12A_2^2,\dots, \tfrac12 A_2^2\right]
\end{equation}
and $A=\diag[A_1, A_2, \dots, A_2]$,  where $A_1,A_2\in \R$ satisfy $A_1\geq A_2/2>0$.
\end{enumerate}
\end{proposition}

\begin{proof}
(i) is a calculation. 

(ii) ``Only if'': by selecting  the body frame with $f_1$  parallel to  the symmetry axis of the body we have $\mathbb{J}=(J_1,J_2,\dots,J_2)$ with $J_1\geq 0$ and $J_2>0$. 
The statement follows by putting $A_1=(J_1+J_2)/\sqrt{2J_2}$ and $A_2=\sqrt{2J_2}$

``If": suppose that  \eqref{eq:physical inertia} holds and $\I$ is a physical inertia tensor with corresponding diagonal matrix $\J$.
 Then there are real numbers $A_1,\dots, A_n$ such that for each $i\neq j$, 
$$ J_i+J_j=A_iA_j.\eqno(E_{ij})$$
We claim that if the $J_i$ satisfy these conditions, then at least $n-1$ of the $A_i$ coincide. In that case suppose (w.l.o.g.) $A_2=\cdots=A_n$. Then comparing the equations $E_{1i}$  shows that all $J_1+J_i$ coincide for $i>1$, and hence $J_2=\cdots=J_n$ as required. 

To prove the claim, note that for any collection $\{i,j,k,\ell\}$ of 4 distinct indices,
$$E_{ij}+E_{ik}-E_{jk} \text{ becomes } 2J_i = A_iA_j+A_iA_k-A_jA_k,$$
and similarly
$$E_{ij}+E_{i\ell}-E_{j\ell} \text{ becomes }  2J_i = A_iA_j+A_iA_\ell-A_jA_\ell.$$
Subtracting these shows
$$(A_i-A_j)(A_k-A_\ell)=0.$$
Suppose $J_1\neq J_2$ and apply this to the indices $1,2,i,j$ with $i,j>2$. It follows that $A_3=A_4=\dots= A_n$. Finally, if $A_1=A_3$ we are done, while if $A_1\neq A_3$ apply the previous reasoning to the indices $1,3,2,4$ and one can conclude that $A_2=A_4$, and hence indeed $A_2=\cdots=A_n$. 
\end{proof}

The usefulness of a condition of type \eqref{eq:inertia-tensor} for the developments in  \cite{FedJov,FedJov2} (also \cite{Jovan,JovaRubber}) seems to arise from the fact that an inertia tensor
satisfying  \eqref{eq:inertia-tensor} maps the set of rank 2 matrices in $\so(n)$ into itself.  For interest, although we make no  use of this, we show that for $n\geq4$, the only physical inertia tensors with this property are axisymmetric. 

\begin{proposition} 
Let $\I$ be a physical inertia tensor defined by \eqref{eq:Phys-Inertia-Intro}. If $n\geq 4$, then $\I$ maps the set of rank two matrices in $\so(n)$ into itself if and only if the body is axisymmetric. 
\end{proposition}

\begin{proof}
Let us consider the case $n=4$. Suppose $\I$ maps the space of rank two matrices in $\so(4)$ into itself and that the body is not axisymmetric. 
Choose the body frame so 
that $\J=\diag[J_1, \dots, J_4]$ with $J_1\neq J_3$, $J_2\neq J_4$. 
A direct calculation shows that
\begin{equation*}
\det(\I((f_1+f_3)\wedge (f_2+f_4)))=(J_1-J_3)^2(J_2-J_4)^2\neq 0,
\end{equation*}
which shows $\I((f_1+f_3)\wedge (f_2+f_4))$ has rank 4, reaching a contradiction. A similar argument shows this implication for $n>4$. 

The converse statement follows from part (ii) of the Proposition\,\ref{P:phys-inertia}.
\end{proof}

\noindent \textbf{Acknowledgements:} This research was made possible by a Newton Advanced Fellowship from the Royal Society, ref: NA140017.  LGN and JM are grateful to the hospitality of the  Department of Mathematics Tullio Levi-Civita of the Univeristy of Padova, during its 2018 intensive period ``Hamiltonian Systems''. 
LGN acknowledges the Alexander Von Humboldt Foundation for a Georg Forster Research Fellowship that funded a  visit to TU Berlin where the last part
of this work was completed.

The authors are thankful to Bo\v{z}idar Jovanovi\'c for comments on an early draft of this paper and for sharing the recent preprint \cite{Gajic} with us.

\vskip 1cm

FF:  Dipartimento di Matematica `Tullio Levi Civita', Universit\`a di Padova, Italy.  fasso@math.unipd.it

LGN: IIMAS, UNAM, Mexico City, Mexico. luis@mym.iimas.unam.mx
 
JM: School of Mathematics, University of Manchester, UK.  j.montaldi@manchester.ac.uk

\end{document}